\documentclass[10pt]{article}
\usepackage{amsmath,amsthm,amssymb}
\usepackage[pdftex]{graphicx}
\usepackage[hang]{caption2}
\usepackage{indentfirst}
\usepackage[section]{placeins}
\numberwithin{equation}{section}
\usepackage{graphics}
\usepackage{pdfsync}
\usepackage[backref]{hyperref}
\usepackage{color}
\usepackage{tikz}
\usepackage{pgfplots}
\usepgfplotslibrary{fillbetween}
\usetikzlibrary{patterns,positioning}
\usetikzlibrary{intersections}
\usetikzlibrary{backgrounds,patterns}
\usetikzlibrary{tikzmark,calc,fadings,arrows,shapes,decorations.pathreplacing}
\usetikzlibrary{shapes,positioning,shadings}

\usepackage{bm}
\usepackage{bbm}

\def\accentsfrancais{applemac}
  \RequirePackage[\accentsfrancais]{inputenc}

\newtheorem{thm}{Theorem}
\newtheorem{theorem}[thm]{Theorem}
\numberwithin{thm}{section}
\newtheorem{lemma}[thm]{Lemma}

\newtheorem{proposition}[thm]{Proposition}
\newtheorem{rem}[thm]{Remark}
\newtheorem{remark}[thm]{Remark}

\newtheorem{definition}[thm]{Definition}

\def\ds{\displaystyle}
\def\emptyset{/\kern-.51em o}
\def\eq{\mathop{\vrule height2,6pt depth-2,3pt
         width -1pt\kern 0pt =}}

\let\norbali\normalbaselines
\def\anorbali{\norbali\advance\lineskip\jot
\advance\baselineskip\jot\advance\lineskiplimit\jot}
\def\ouvre{\let\normalbaselines\anorbali}
\def\R{{\mathbb{R}}}

\def\Pc{{\mathcal{P}}}

\def\Uc{{\mathcal{U}}}
\def\Vc{{\mathcal{V}}}
\def\Wc{{\mathcal{W}}}
\def\Rc{{\mathcal{R}}}
\def\Zc{{\mathcal{Z}}}
\def\Yc{{\mathcal{Y}}}
\def\wc{{\;\rightharpoonup\;}}

\def\N{{\mathbb{N}}}

\def\Ge{{\bf e}}

\def\Gn{{\bf n}}

\def\GD{{\bf D}}
\def\GE{{\bf E}}

\def\GG{{\bf G}}
\def\GH{{\bf H}}
\def\GI{{\bf I}}

\def\GR{{\bf R}}
\def\GS{{\bf S}}

\def\GU{{\bf U}}

\def\GV{{\bf V}}

\def\UU{\mathbb{U}}
\def\VV{\mathbb{V}}

\def\e{\varepsilon}

\def\Te{{\cal T}_\varepsilon}

\def\TRe{{\mathfrak{T}_\e}}

\def\O{\Omega}
\def\o{\omega}
\def\wh{\widehat }
\def\wt{\widetilde }

\title{Asymptotic Behavior for Textiles in von-K{\'a}rm{\'a}n regime}
\author{Georges Griso\footnote{Sorbonne Universit\'e, CNRS, Universit\'e de Paris, Laboratoire Jacques-Louis Lions (LJLL), F-75005 Paris, France, griso@ljll.math.upmc.fr}, Julia Orlik\footnote{Fraunhofer ITWM, 67663 Kaiserslautern, Germany, orlik@itwm.fhg.de}, Stephan Wackerle \footnote{Fraunhofer ITWM, 67663 Kaiserslautern, Germany, wackerle@itwm.fhg.de}}

\begin{document}
	\maketitle
	\begin{abstract}
		
		The paper is dedicated to the investigation of simultaneous homogenization and dimension reduction of textile structures as elasticity problem with an energy in the von-K\'{a}rm\'{a}n-regime. An extension for deformations is presented allowing to use the decomposition of plate-displacements. The limit problem in terms of displacements is derived with the help of the unfolding operator and yields in the limit the von-K\'{a}rm\'{a}n plate with linear elastic cell-problems. It is shown, that for homogeneous isotropic beams in the structure, the resulting plate is orthotropic. As application of the obtained limit plate we study the buckling behavior of orthotropic textiles.
	\end{abstract}

\section{Introduction}

In contrast to our first paper about homogenization of textiles \cite{KT}, where a geometrical linear elasticity is considered, we investigate here textiles with von-K\'{a}rm\'{a}n energy. The von-K\'{a}rm\'{a}n model is a nonlinear plate model, which is stated with respect to displacements and widely used by mathematicians and engineers, see \cite{Shell1,Shell2,CiarletKarman,ciarlet1997mathematical,FJMKarman,Friesecke06,Puntel}. To achieve the von-K\'{a}rm\'{a}n model in the limit we consider an elastic energy of order
$||e(u_\e)||_{L^2(\O^*_\e)} \le C\e^{5/2}$, which is in consensus with \cite{BCM,CiarletKarman,FJMKarman,Friesecke06}. 
A simultaneous homogenization and dimension reduction of a von-K\'{a}rm\'{a}n plate was already studied in \cite{NV}, however in our paper we give the corrector results related on the periodic topology of the textile.

Since the von-K\'{a}rm\'{a}n plate arises as $\Gamma$-Limit from geometrical nonlinear problems, it is necessary to consider initially deformations. Hence, the homogenization of the textile for von-K\'{a}rm\'{a}n energy begins with the extension of a deformation into the holes of the structure. This extension is applied onto the deformations of the textile beam structure for glued beams. Due to the fact that the limit-plate is stated with respect to displacements we directly introduce the decomposition of the displacement associated to the extended deformation, see \cite{Shell1,GStRods,GDecomp}. For the elementary displacements we establish the Korn-type estimates giving rise to the asymptotic behavior of the fields. To derive the homogenized model the unfolding and rescaling operator (see for instance \cite{CDGHoles,CDG,GHJ,KT}) accounting for both homogenization and dimension reduction is used. 
For the von-K\'{a}rm\'{a}n-plate studied here it is necessary to investigate also the nonlinear term in the Green-St.Venant strain tensor. The additional term yield the von-K\'{a}rm\'{a}n nonlinearities in the limit.
The derived asymptotic limits allow to prove with arguments of $\Gamma$-convergence to show that the limit energy in in fact of von-K\'{a}rm\'{a}n-type. It is proven, that the homogenized limit energy admits minima. The uniqueness, though, is not provable.

Although
the initial and the homogenized problem are nonlinear, the cell-problems for the von-K\'{a}rm\'{a}n
plate are linear and in fact the same as achieved for a linear elastic plate. Furthermore, the cell-problems yield for isotropic homogeneous beams an orthotropic plate, which is valid for the von-K\'{a}rm\'{a}n energy and linear elasticity.

The final part of the paper is devoted to homogenization of the pre-stress in yarns and modeling of the buckling of the von-K\'{a}rm\'{a}n plate under pre-strain, like in \cite{LM,CM,Puntel,berd}, but for an orthotropic plate. 

\section{ Preliminary extension results for deformations and displacements}\label{sec:prelim_ext}
In this section,  $\O$ and $\O'$ are two bounded domains in $\R^n$ containing the origin with Lipschitz boundaries and such that $\O\subset \O'$. For every $\e>0$, we denote $\O_\e=\e\,\O$ and $\O'_\e=\e\, \O'$. In Lemma \ref{L1} we prove an extension result for deformations in $H^1(\O_\e)^n$.\\[1mm]
The lemma below is based on the rigidity theorem  obtained by  G. Friesecke, R. James and S. M{\"{u}}ller in \cite{RigidityTH}. Here, for the starshaped open sets with respect to a ball, we use a variant of this theorem which explicitly gives the dependence of the constants in the estimates in terms of only two parameters which depend on the geometry of the domain: its diameter and the radius of the ball (see \cite{bg2,G-DMT}).
	\begin{lemma}\label{L1}  For every deformation $v$ in $H^1(\O_\e)^n$ there exists a deformation $\wt v$ in $H^1(\O'_{\e})^n$ satisfying
		\begin{equation}\label{ExtensionL1}
		\begin{aligned}
		&\wt{v}_{|\O_\e}=v,\\
		&\big\|\hbox{dist}\big(\wt v, SO(n)\big)\big\|_{L^2(\O'_{\e})}\leq C\big\|\hbox{dist}(v, SO(n))\big\|_{L^2(\O_\e)}.
		\end{aligned}
		\end{equation}		The constant does not depend on $\e$.
	\end{lemma}
\begin{proof} First, some classical recalls and then the proof. 
\begin{itemize}
\item (i) Since $\O$ is a bounded domain with Lipschitz boundary, there exist $N\in \N^*$, $R_1$ and $R_2$ two strictly positive constants and  a finite set $\{{\cal O}_1, \ldots, {\cal O}_N\}$ of open subsets of $\O$, each of diameter less than $R_1$ and  starshaped with respect to a  ball of radius $R_2$ ($B(A_i, R_2)$, $A_i\in {\cal O}_i$)  such that
$$\O=\bigcup_{k=1}^N{\cal O}_k.$$ As a consequence, there exists $r$ such that for every ${\cal O}_i$, $i\in\{1,\ldots,N\}$ there exists a chain from ${\cal O}_1$ to ${\cal O}_i$
$${\cal O}_{l_1}={\cal O}_1, \;\; {\cal O}_{l_2},\;\;\ldots,  \;\; {\cal O}_{l_p}={\cal O}_i,\qquad p\in \{1,\ldots,N\}$$ such that, if $p>1$ one has  ${\cal O}_{l_j}\cap {\cal O}_{l_{j+1}}$, $j\in\{1,\ldots, p-1\}$,  contains a ball of radius $r$.
\item (ii) Let ${\cal O}$ be an open set in $\R^n$ included in the  ball $B(A; R_1)$  and starshaped with respect to the ball  $B(A,R_2)$, $R_1>0,\; R_2>0$. Theorem II.1.1 in \cite{bg2}  claims that for every deformation $v\in H^1({\cal O})^n$, there exist a matrix ${\bf R}\in SO(n)$ and ${\bf a}\in \R^n$ such that
$$
\|v-{\bf a}-\GR\, x\|_{L^2({\cal O})}\leq CR_1 \|dist(\nabla v, SO(3))\|_{L^2({\cal O})},\quad \|\nabla v-\GR\|_{L^2({\cal O})}\leq C\|dist(\nabla v, SO(n))\|_{L^2({\cal O})}.
$$ The constant $C$ depends only on $\ds {R_1\over R_2}$ and $n$.\\[1mm]
Transform ${\cal O}$ by a dilation of ratio $\e>0$ and center $A$, the above result gives: for every deformation $v\in H^1({\cal O}_\e)^n$ where ${\cal O}_\e\doteq \e {\cal O}$, there exist a matrix ${\bf R}\in SO(n)$ and ${\bf a}\in \R^n$ such that
$$
\|v-{\bf a}-\GR\, x\|_{L^2({\cal O}_\e)}\leq C \e  \|dist(\nabla v, SO(n))\|_{L^2({\cal O}_\e)},\quad \|\nabla v-\GR\|_{L^2({\cal O}_\e)}\leq C \|dist(\nabla v, SO(n))\|_{L^2({\cal O}_\e)}.
$$ The constant $C$ does not depend on $\e$.
\item  (iii) $\O$ and $\O'$  being two bounded domains in $\R^n$ with Lipschitz boundaries and such that $\O\subset \O'$. There exists a continuous linear extension operator ${\cal P}'$ from $H^1(\O)^n$ into $H^1(\O')^n$ satisfying
$$
\forall v\in H^1(\O)^n,\quad {\cal P}'(v)_{|\O}=v,\quad \|{\cal P}'(v)\|_{L^2(\O')}\leq C\|v\|_{L^2(\O)},\quad \|{\cal P}(v)\|_{H^1(\O')}\leq C\|v\|_{H^1(\O)}.
$$ If we transform $\O$ and $\O'$ by the same dilation of ratio $\e$ (and center $O\in \O$), this extension operator induces an extension operator ${\cal P}'_\e$ from $H^1(\O_\e)^3$ into $H^1(\O'_\e)^3$ satisfying
$$\forall v\in H^1(\O_\e)^3,\quad 
\left\{\begin{aligned}
&{\cal P}'_\e(v)_{|\O_\e}=v,\qquad \|{\cal P}'_\e(v)\|_{L^2(\O'_\e)}\leq C\|v\|_{L^2(\O_\e)},\\
&\|{\cal P}'_\e(v)\|_{L^2(\O'_\e)}+\e \|\nabla {\cal P}'_\e(v)\|_{L^2(\O'_\e)}\leq C\big(\|v\|_{L^2(\O_\e)}+\e \|\nabla v\|_{L^2(\O_\e)}\big).
\end{aligned}\right.
$$ The constants do not depend on $\e$.
\end{itemize}
Now, consider a deformation $v\in H^1(\O_\e)^n$. We apply (ii) with the open sets ${\cal O}_{i,\e}=A_i+\e ({\cal O}_i-A_i)$, there exist  matrices ${\bf R}_i\in SO(n)$ and  vectors ${\bf a}_i\in \R^n$ such that
$$
\begin{aligned}
&\|v-{\bf a}_i-\GR_i\, x\|_{L^2(\e{\cal O}_{i,\e})}\leq C \e  \|dist(\nabla v, SO(n))\|_{L^2(\e {\cal O}_{i,\e})},\\
&\|\nabla v-\GR_i\|_{L^2(\e {\cal O}_{i,\e})}\leq C \|dist(\nabla v, SO(n))\|_{L^2(\e{\cal O}_{i,\e})}.
\end{aligned}
$$ The constant $C$ does not depend on $\e$.\\[1mm]
 Then, using the second part of (i), we compare ${\bf R}_i$ to ${\bf R}_1$ as well as  ${\bf a}_i$ to ${\bf a}_1$, $i\in\{1,\ldots,N\}$. As a consequence, one obtains that
$$
\|v-{\bf a}_1-\GR_1\, x\|_{L^2(\O_\e)}\leq C \e  \|dist(\nabla v, SO(3))\|_{L^2(\O_\e)},\quad \|\nabla v-\GR_1\|_{L^2(\O_e)}\leq C \|dist(\nabla v, SO(n))\|_{L^2(\O_\e)}.
$$ The constants do not depend on $\e$.\\[1mm]
Now, we define the extension of $v$. We set
$$
\wt{v}={\cal P}'_\e(v-{\bf a}_1-\GR_1\, x)+{\bf a}_1+\GR_1\, x\qquad \hbox{a.e. in }\O'_\e.
$$	We easily check \eqref{ExtensionL1}.	
\end{proof}

\section{The structure}
\subsection{ Parameterization of the yarns}
To see the parametrization of yarns and the structure we refer to \cite{KT}. Nevertheless, we shortly repeat the most important definitions and results.
The middle line of a beam is paramtrized by rescaled function $\ds \Phi_\e = \e \Phi(\frac{z}{\e})$ of
\begin{align}
\Phi(z) = \begin{cases}
-\kappa, &\hbox{if } z \in [0,\kappa],\\
\kappa\Big(6{(z -\kappa)^2\over (1-2\kappa)^2}-4{(z - \kappa)^3\over (1 - 2\kappa )^3} - 1 \Big) & \hbox{if } z\in [\kappa, 1-\kappa],\\
\kappa & \hbox{if } z\in [1-\kappa,1],\\
\Phi(2-z) &\hbox{if } z\in [1,2].
\end{cases}
\end{align}
Then the beams in the structure are defined by
\begin{align*}
&P^{(1)}_r\doteq \bigl\{z\in \R^3\;|\; z_1\in (0,L),\enskip (z_2,z_3)\in  (-\kappa \e,\kappa \e)^2\big\},\qquad P^{(2)}_r\doteq  \bigl\{z\in \R^3\;|\; z_2\in (0,L),\enskip (z_1,z_3)\in  (-\kappa 
\e,\kappa \e)^2\big\}.
\end{align*}
for the reference beams in the two directions. Then the curved beams are defined by
\begin{equation*}
\begin{aligned}
& \Pc^{(1,q)}_{\e}\doteq\Big\{x\in \R^3\; |\; x = \psi_\e^{(1,q)}(z),\enskip z\in P^{(1)}_r\Big\},\qquad  \Pc^{(2,p)}_{\e}\doteq\Big\{x\in \R^3\; |\; x = \psi_\e^{(2,p)}(z),\enskip z\in P^{(2)}_r\Big\},\\
\end{aligned}
\end{equation*}
with the diffeomorphisms
\begin{align*}
&\psi_\e^{(1,q)}(z) \doteq M^{(1,q)}_\e(z_1) + z_2\Ge_2 + z_3\Gn_\e^{(1,q)}(z_1),\qquad\psi_\e^{(2,p)}(z) \doteq M^{(2,p)}_\e(z_2) + z_1\Ge_1 + z_3\Gn_\e^{(2,p)}(z_2),
\end{align*}
and the corresponding middle lines
\begin{equation*}
\begin{aligned}
&M^{(1,q)}_\e(z_1) \doteq z_1\Ge_1 + q\e\Ge_2 + (-1)^{q+1}\Phi_{\e}(z_1)\Ge_3,\qquad M^{(2,p)}_\e(z_2) \doteq p\e\Ge_1 + z_2\Ge_2 + (-1)^p \Phi_{\e}(z_2)\Ge_3.
\end{aligned}
\end{equation*}

\subsection{Parameterization of the whole structure}

Denote $\O^*_{\e}$ the whole structure (see \cite{KT} for details)
\begin{equation}\label{STR}
\O^*_{\e}\doteq\O_\e\cap\Big(\bigcup_{p=0}^{2N_\e}\Pc^{(1,q)}_{\e}\cup\bigcup_{q=0}^{2N_\e}\Pc^{(2,p)}_{\e}\Big),\qquad \O_\e\doteq\o\times (-2\kappa \e,2\kappa\e),\qquad \o=(0,L)^2.
\end{equation} 

\begin{figure}\centering
		\begin{tikzpicture}[remember picture,
		x={(1cm,0cm)},
		y={(0cm,1cm)},
		z={({0.5*cos(45)},{0.5*sin(45)})},
		fill opacity = 1]
		

		\begin{scope}[thick]

		\draw[fill = gray] (0,0,5)--(0,0,6)--(0,1,6)--(0,1,5)--cycle;
		\draw[name path = frontlow]  (6,1,5)--(5,1,5) to[out=180,in=0] (1,0,5) -- (0,0,5) ;
		\draw[name path = fronthigh]  (6,2,5)--(5,2,5) to[out=180,in=0] (1,1,5) -- (0,1,5);
		\draw[name path = backlow]   (6,1,6)--(5,1,6) to[out=180,in=0] (1,0,6) -- (0,0,6) ;
		\draw[name path = backhigh] (6,2,6)--(5,2,6) to[out=180,in=0] (1,1,6) -- (0,1,6) ;
		
		\tikzfillbetween[of=frontlow and fronthigh]{ left color=gray!20!white, right color = gray, shading angle=60,};
		\tikzfillbetween[of=fronthigh and backhigh]{ left color=gray, right color = gray!20!white, shading angle=180,};
		
		\draw[name path = frontlow] (6,1,5)--(5,1,5) to[out=180,in=0] (1,0,5) -- (0,0,5) ;
		\draw[name path = fronthigh] (6,2,5)--(5,2,5) to[out=180,in=0] (1,1,5) -- (0,1,5 );
		\draw[name path = backhigh]  (6,2,6)--(5,2,6) to[out=180,in=0] (1,1,6) -- (0,1,6) ;
		
		%
		\draw[fill = gray] (6,0,5)--(6,0,6)--(6,1,6)--(6,1,5)--cycle;
		\end{scope}

		\begin{scope}[thick,
		]
		
		\draw[fill=gray] (1,2,6)--(1,1,6)--(0,1,6)--(0,2,6)--cycle;
		\draw[name path = righthigh] (1,1,0)--(1,1,1) to[out=55,in=-135] (1,2,5)--(1,2,6) ;
		\draw[name path = rightlow] (1,0,0)--(1,0,1) to[out=45,in=-125] (1,1,5)--(1,1,6) ;
		\draw[name path = lefthigh] (0,1,0)--(0,1,1) to[out=55,in=-135] (0,2,5)--(0,2,6) ;
		\draw[name path = leftlow] (0,0,0)--(0,0,1) to[out=45,in=-125] (0,1,5)--(0,1,6) ;
		\draw[fill=gray] (0,0,0)--(0,1,0)--(1,1,0)--(1,0,0)--cycle;
		
		\tikzfillbetween[of=rightlow and righthigh]{ left color=gray!20!white, right color = gray, shading angle=60,};
		\tikzfillbetween[of=lefthigh and righthigh]{ left color=gray, right color = gray!20!white, shading angle=180,};
		
		\draw[name path = righthigh] (1,1,0)--(1,1,1) to[out=55,in=-135] (1,2,5)--(1,2,6) ;
		\draw[name path = rightlow] (1,0,0)--(1,0,1) to[out=45,in=-125] (1,1,5)--(1,1,6) ;
		\draw[name path = lefthigh] (0,1,0)--(0,1,1) to[out=55,in=-135] (0,2,5)--(0,2,6) ;
		\draw[] (0,2,6)--(1,2,6)--(1,1,6);
		\end{scope}

		\begin{scope}[thick,
		]
		\draw[name path = frontlow]  (0,1,0)--(1,1,0) to[out=0,in=180] (5,0,0) -- (6,0,0) ;
		\draw[name path = fronthigh] (0,2,0)--(1,2,0) to[out=0,in=180] (5,1,0) -- (6,1,0) ;
		\draw[name path = backlow]    (0,1,1)--(1,1,1) to[out=0,in=180] (5,0,1) -- (6,0,1) ;
		\draw[name path = backhigh]   (0,2,1)--(1,2,1) to[out=0,in=180] (5,1,1) -- (6,1,1);
		
		\draw[fill=gray ] (0,1,0)--(0,2,0)--(0,2,1)--(0,1,1)--cycle;
		
		\tikzfillbetween[of=frontlow and fronthigh]{ left color=gray!20!white, right color = gray, shading angle=60,};
		\tikzfillbetween[of=fronthigh and backhigh]{ left color=gray, right color = gray!20!white, shading angle=180,};

		\draw[name path = frontlow]   (0,1,0)--(1,1,0) to[out=0,in=180] (5,0,0) -- (6,0,0) ;
		\draw[name path = fronthigh]  (0,2,0)--(1,2,0) to[out=0,in=180] (5,1,0) -- (6,1,0) ;
		\draw[name path = backhigh]  (0,2,1)--(1,2,1) to[out=0,in=180] (5,1,1) -- (6,1,1);
		
		%
		\draw[fill=gray] (6,0,0)--(6,1,0)--(6,1,1)--(6,0,1)--cycle;
		\end{scope}

		\begin{scope}[thick]
		
		\draw[] (6,0,6)--(6,1,6)-- (5,1,6)-- (5,0,6)--cycle;
		
		\draw[name path = rightlow] (6,1,0)--(6,1,1) to[out=45,in=-145] (6,0,5) -- (6,0,6);
		\draw[name path = righthigh] (6,2,0)--(6,2,1) to[out=35,in=-145] (6,1,5) -- (6,1,6);
		\draw[name path = leftlow] (5,1,0)--(5,1,1) to[out=45,in=-135] (5,0,5) -- (5,0,6);
		\draw[name path = lefthigh] (5,2,0)--(5,2,1) to[out=35,in=-145] (5,1,5) -- (5,1,6);

		\tikzfillbetween[of=rightlow and righthigh]{ left color=gray!20!white, right color = gray, middle color=black!40!gray,};
		\tikzfillbetween[of=righthigh and lefthigh]{ top color=gray, bottom color = gray!20!white, middle color=black, shading angle=0};
		
		\draw[name path = rightlow] (6,1,0)--(6,1,1) to[out=45,in=-145] (6,0,5) -- (6,0,6);
		\draw[name path = righthigh] (6,2,0)--(6,2,1) to[out=35,in=-145] (6,1,5) -- (6,1,6);
		\draw[name path = lefthigh] (5,2,0)--(5,2,1) to[out=35,in=-145] (5,1,5) -- (5,1,6);

		\draw[fill=gray] (5,1,0)--(6,1,0)--(6,2,0)--(5,2,0)--cycle;
		\end{scope}

		\begin{scope}[thick]
		\clip[draw=none](4,0,5)--(7,0,5)--(7,3,5)--(4,3,5)--cycle; 
		
		\draw[fill = gray] (0,0,5)--(0,0,6)--(0,1,6)--(0,1,5)--cycle;
		\draw[name path = frontlow]  (6,1,5)--(5,1,5) to[out=180,in=0] (1,0,5) -- (0,0,5) ;
		\draw[name path = fronthigh]  (6,2,5)--(5,2,5) to[out=180,in=0] (1,1,5) -- (0,1,5);
		\draw[name path = backlow]   (6,1,6)--(5,1,6) to[out=180,in=0] (1,0,6) -- (0,0,6) ;
		\draw[name path = backhigh] (6,2,6)--(5,2,6) to[out=180,in=0] (1,1,6) -- (0,1,6) ;
		
		\tikzfillbetween[of=frontlow and fronthigh]{ left color=gray!20!white, right color = gray, shading angle=60,};
		\tikzfillbetween[of=fronthigh and backhigh]{ left color=gray, right color = gray!20!white, shading angle=180,};
		
		\draw[name path = frontlow] (6,1,5)--(5,1,5) to[out=180,in=0] (1,0,5) -- (0,0,5) ;
		\draw[name path = fronthigh] (6,2,5)--(5,2,5) to[out=180,in=0] (1,1,5) -- (0,1,5 );
		\draw[name path = backhigh]  (6,2,6)--(5,2,6) to[out=180,in=0] (1,1,6) -- (0,1,6) ;

		\draw[fill = gray] (6,1,5)--(6,1,6)--(6,2,6)--(6,2,5)--cycle;
		
		\end{scope}
		
		\begin{scope}
		\draw[|-|,thick] (0.5,-0.6,0) -- (5.5,-0.6,0) node [midway, below, sloped] (TextNode1) {\Large $1$};
		\draw[dashed,thick] (0.5,-0.6,0) -- (0.5,0.5,0);
		\draw[dashed,thick] (5.5,-0.6,0) -- (5.5,1.5,0);
		
		\draw[|-|,thick] (0,-0.16,0) -- (1,-0.16,0) node [pos=.25,yshift=-1pt,xshift=-2pt, below, sloped] (TextNode2) {\Large $2\kappa$};
		\draw[|-|,thick] (-0.16,0,0) -- (-0.16,1,0) node [midway, above, sloped] (TextNode3) {\Large $2\kappa$};
		
		\draw[|-|,thick] (6,-0.5,0.5) -- (6,-0.5,5.5) node [midway, below, sloped] (TextNode4) {\Large $1$};
		\draw[dashed,thick] (6,-0.5,0.5)-- (6,0.5,0.5);
		\draw[dashed,thick] (6,-0.5,5.5)-- (6,1.5,5.5);
		\end{scope}

		\end{tikzpicture}
		\caption{The domain $Y^*\subset Y=(0,1)^2\times (-2\kappa,2\kappa)$, a quarter of the periodicity cell of the full structure.}\label{fig:ystar}
\end{figure}
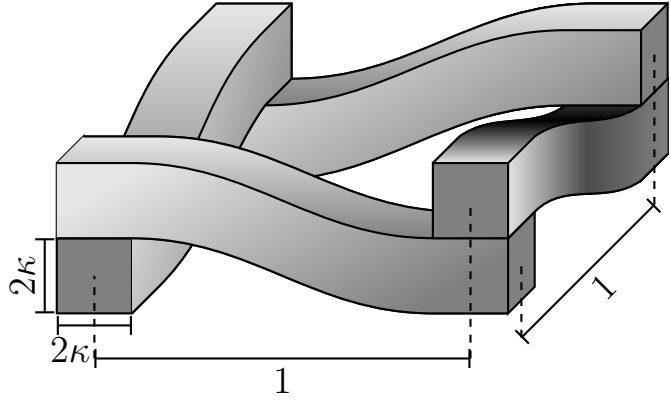

\subsection{An extension result}
	The presented extension heavily depends on the fact that the beams are glued. For a more general contact condition as in \cite{KT} it is necessary to treat the two directions separately and obtain two deformations, which give the same limit for $g_\e \sim \e^4$. Nevertheless, the more general case would exceed the bounds of this paper.

\begin{proposition}\label{P1} For every deformation $v$ in $H^1(\O^*_{\e})^3$ there exists a deformation $\wt{v}$ in $H^1(\O_{\e})^3$ satisfying
		\begin{equation}\label{Extension}
		\begin{aligned}
		&\wt{v}_{|\O^*_{\e}}=v,\\
		&\big\|\hbox{dist}\big(\wt{v}, SO(3)\big)\big\|_{L^2(\O_{\e})}\leq C\big\|\hbox{dist}(v, SO(3))\big\|_{L^2(\O^*_{\e})}.
		\end{aligned}
		\end{equation}
		The constant does not depend on $\e$.
	\end{proposition}

	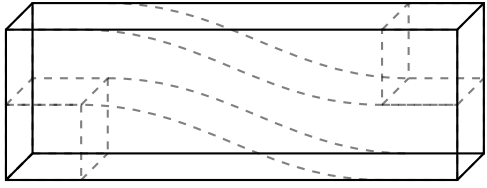
\begin{figure}[htb]\centering
		\begin{tikzpicture}[remember picture,
		x={(1cm,0cm)},
		y={(0cm,1cm)},
		z={({0.5*cos(45)},{0.5*sin(45)})}]
		
		\begin{scope}[opacity=0.5]
		\draw[thick,dashed] (0,0,0)--(1,0,0)--(1,1,0)--(0,1,0)--cycle;
		\draw[thick,dashed] (5,1,0)--(6,1,0)--(6,2,0)--(5,2,0)--cycle;
		\draw[thick,dashed] (0,1,0)--(1,1,0) to[out=0,in=180] (5,0,0) -- (6,0,0);
		\draw[thick,dashed] (0,2,0)--(1,2,0) to[out=0,in=180] (5,1,0) -- (6,1,0);
		\draw[thick,dashed] (0,1,0)--(0,2,0);
		\draw[thick,dashed] (6,0,0)--(6,1,0);
		\end{scope}
		\begin{scope}[opacity=0.5]
		\draw[thick,dashed] (0,0,1)--(1,0,1)--(1,1,1)--(0,1,1)--cycle;
		\draw[thick,dashed] (5,1,1)--(6,1,1)--(6,2,1)--(5,2,1)--cycle;
		\draw[thick,dashed] (1,1,1) to[out=0,in=180] (5,0,1) ;
		\draw[thick,dashed] (1,2,1) to[out=0,in=180] (5,1,1) ;
		\draw[thick,dashed] (0,1,1)--(0,2,1);
		\draw[thick,dashed] (6,0,1)--(6,1,1);	
		\end{scope}
		\begin{scope}[opacity=0.5]
		\draw[thick,dashed] (0,0,0)--(0,0,1);
		\draw[thick,dashed] (1,0,0)--(1,0,1);
		\draw[thick,dashed] (1,1,0)--(1,1,1);
		\draw[thick,dashed] (0,1,0)--(0,1,1);
		
		\draw[thick,dashed] (5,1,0)--(5,1,1);
		\draw[thick,dashed] (6,1,0)--(6,1,1);
		\draw[thick,dashed] (6,2,0)--(6,2,1);
		\draw[thick,dashed] (5,2,0)--(5,2,1);
		
		\draw[thick,dashed] (0,2,0)--(0,2,1);
		\draw[thick,dashed] (6,0,0)--(6,0,1);
		\draw[thick,dashed] (0,0,0)--(0,0,1);
		\draw[thick,dashed] (0,0,0)--(0,0,1);
		\end{scope}
		\begin{scope}
		\draw[thick] (0,0,0) rectangle (6,2,0);
		\draw[thick] (0,0,1) rectangle (6,2,1);
		\draw[thick] (0,0,0)--(0,0,1);
		\draw[thick] (6,0,0)--(6,0,1);
		
		\draw[thick] (0,2,0)--(0,2,1);
		\draw[thick] (6,2,0)--(6,2,1);
		\end{scope}
		
		\end{tikzpicture}
		\caption{First extension domain}\label{fig:extension1}
	\end{figure}
	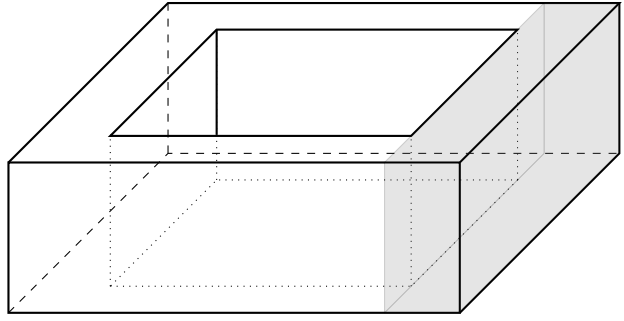
\begin{figure}[htb]\centering
		\begin{tikzpicture}[remember picture,
		x={(1cm,0cm)},
		y={(0cm,1cm)},
		z={({0.5*cos(45)},{0.5*sin(45)})}]
		
		\begin{scope}
		\draw[thick] (0,0,0)--(6,0,0)--(6,2,0)--(0,2,0)--cycle;
		\draw[thick] (0,2,0)--(0,2,6)--(6,2,6)--(6,2,0);
		\draw[thick] (6,0,0)--(6,0,6)--(6,2,6);
		\draw[thick] (1,2,1)--(5,2,1)--(5,2,5)--(1,2,5)--cycle;
		\end{scope}
		
		\begin{scope}
		\draw[thin,dashed] (0,0,0)--(0,0,6)--(6,0,6);
		\draw[thin,dashed] (0,0,6)--(0,2,6);
		\end{scope}
		
		\begin{scope}
		\draw[thin,dotted] (1,0,1)--(5,0,1)--(5,0,5)--(1,0,5)--cycle;
		\draw[thin,dotted] (1,0,1)--(1,2,1);
		\draw[thin,dotted] (5,0,1)--(5,2,1);
		\draw[thin,dotted] (5,0,5)--(5,2,5);
		\draw[thin,dotted] (1,0,5)--(1,2,5);
		\draw[thick] (1,0.6,5)--(1,2,5);
		\end{scope}
		
		\begin{scope}
		\draw[fill=gray,opacity=0.2] (6,0,0)--(6,0,6)--(6,2,6)--(6,2,0)--cycle;
		\draw[fill=gray,opacity=0.2] (5,0,0)--(6,0,0)--(6,2,0)--(5,2,0)--cycle;
		\draw[fill=gray,opacity=0.2] (5,2,0)--(6,2,0)--(6,2,6)--(5,2,6)--cycle;
		\draw[opacity=0.2] (5,0,0)--(5,0,6)--(5,2,6);
		\end{scope}
		
		\end{tikzpicture}
		\caption{Periodicity cell of the periodic plate with holes.}\label{fig:extension2}
	\end{figure}

\begin{proof} Now that the general extension for Lipschitz domains in nonlinear elasticity is recalled in the above lemma, we specify the extension procedure for the the domain $\O^*_{\e}$.
		
		First, we divide the domain $\O^*_{\e}$ into portions included in domains isometric to the parallelotope $(0, \e+2\kappa\e)\times(0,2\kappa\e)\times(0,4\kappa\e)\footnotemark$\footnotetext{We reduce the parallelotopes that are in contact with the boundary of $\o$.} as depicted in Figure \ref{fig:extension1}. These portions include a curved beam  and parts of the beams in the perpendicular direction with which the beam is in contact. Besides, after a rotation and/or a reflection,  all the portions are of the same form and itself Lipschitz-domains. Furthermore, note that these portions intersect each other and every contact cylinder $C_{pq}\times(-2\kappa\e,2\kappa\e)$ with $C_{pq} = (p\e-\kappa\e,p\e+\kappa\e)\times(q\e-\kappa\e,q\e+\kappa\e)$ is used in four of such domains.
		
		Since every portion is a Lipschitz domain the  extension procedure given in Lemma \ref{L1} is applicable for every $v\in H^1(\O^*_{\e})$ and yields an extension to the parallelotope (e.g. $(p\e-\kappa\e,(p+1)\e+\kappa\e)\times(q\e-\kappa\e,q\e+\kappa\e)\times(-2\kappa\e,2\kappa\e)$).
		As second step, we define new domains included in $(p\e-\kappa\e,(p+1)\e+\kappa\e)\times(q\e-\kappa\e,(q+1)\e+\kappa\e)\times(-2\kappa\e,2\kappa\e)$ by collecting four of the above portions as depicted in Figure \ref*{fig:extension2}. Note that the contact cylinders in every corner of the new domain is used by two portions. Obviously this domain is again a Lipschitz domain and hence we extend all the fields into the holes using again Lemma \ref{L1}. 
		
		To obtain the full extension we reassemble the structure. To do this, note that the domains $(p\e-\kappa\e,(p+1)\e+\kappa\e)\times(q\e-\kappa\e,(q+1)\e+\kappa\e)\times(-2\kappa\e,2\kappa\e)$ have an overlap. This overlap includes every beam twice and the contact cylinders again fourfold, i.e. the overlap consists exactly of the domains before the last extension. Together with the interportions from the step before we obtain that the contact cylinders are the most used domains for the extension, namely eight times. This influences the estimate and finally give the final extension $\wt{v}$ which satisfies
		\begin{align*}
		&\big\|\hbox{dist}\big(\wt{v}, SO(3)\big)\big\|_{L^2(\O_{\e})}\leq  C  \big\|\hbox{dist}(v, SO(3))\big\|_{L^2(\O^*_{\e})}
		\end{align*}
where the constant does not depend on $\e$. By construction, we have $\wt{v}_{|\O^*_{\e}}=v$.	
\end{proof}		
	Henceforth, we use the extended deformation $v\in H^1(\O_\e)$, which is a deformation of a periodic plate without holes. This allows to use the results in the papers \cite{Shell1} and general results of \cite{GUnfold}, \cite{KT}, \cite{Shell2}.\\[1mm]
The structure is clamped on its lateral boundary. Moreover, in contrast to \cite{KT} here we assume a glued contact, which corresponds to the case $g_\e \equiv 0$ in \cite{KT}. This allows to obtain one deformation field for the whole structure $\O_\e^*$ instead of one for each beam as in $\cite{KT}$.\\[1mm]
Denote
$$\gamma=\partial\o\cap  \{x_2=0\}=(0,L)\times \{0\},\qquad \Gamma_\e=\gamma\times (-2\kappa\e,2\kappa_\e).$$
The set of the admissible deformations are
\begin{equation}
\begin{aligned}\label{defo-space}
&\GV_{\e}\doteq\Big\{v\in H^1(\O^*_{\e})^3\;\;|\; \; \hbox{such that} \;\; v=I_d\quad \hbox{a.e. on } \partial \O^*_{\e}\cap  \Gamma_\e\Big\},\\
&\GD_{\e}\doteq\Big\{v\in H^1(\O_{\e})^3\;\;|\; \; \hbox{such that} \;\; v=I_d\quad \hbox{a.e. on } \Gamma_\e\Big\}.
\end{aligned}
\end{equation}

\begin{remark}\label{Rem1} Every deformation belonging to $\GV_\e$ is extended in $(0,L)\times (-\kappa\e,0)\times (-2\kappa\e,2\kappa_\e)$ by setting  $v=I_d$ in this open set. Then, Proposition \ref{P1} gives an extension of $v$ whose restriction to $\O_\e$ belongs to $\GD_\e$ and satisfies \eqref{Extension}.
\end{remark}

\section{ The non-linear elasticity problem}\label{sec:problem}
Set
$$\Yc':=(0,2)^2,\qquad \Yc:=(0,2)^2\times (-2\kappa, 2\kappa ).$$
Let  $\Yc^*\subset  \Yc$ be the reference cell of the beam structure. The cell $\Yc^*$ is deduced from $Y^*$ (see Figure \ref{fig:ystar}) after two symmetries with respect to the planes $y_1=1$ and $y_2=1$.\\[1mm]
Denote $\wh W$ the local elastic energy density, then the total elastic energy is
\begin{equation}\label{stpr}
\begin{aligned}
{\cal J}_{\e}(v)=\int_{\O^*_{\e}} \wh W_\e\big(\cdot,\nabla v\big)\, dx- \int_{\O^*_{\e}} f_{\e}\cdot(v-I_d)\, dx,\qquad
\forall v\in \GV_\e,
\end{aligned}
\end{equation}
where $I_d$ is the identity map.
The local density energy $\wh W\, :\ \,  \Yc^*\times \GS_3\longrightarrow \R^+\cup\{+\infty\}$ is given by
$$
\widehat{W}_\e(\cdot , F)=\left\{
\begin{aligned}
&Q\Big(\frac{\cdot}{\e} ,{1\over 2}(F^TF-\GI_3)\Big) &&\hbox{if} \quad \det(F)>0,\\
&+\infty  &&\hbox{if} \quad \det(F)\leq 0,
\end{aligned}\right.
$$ where $\GS_3$ is the space of $3\times 3$ symmetric matrices. The quadratic form $Q$ is defined by
$$Q(y,S)=a_{ijkl}(y)S_{ij}S_{kl} \quad \hbox{for a.e. $y\in \Yc^*$ and for all $S\in \GS_3$},$$ 
where the $a_{ijkl}$'s belong to $L^\infty(\Yc^*)$ and are periodic with respect to $\Ge_1$ and $\Ge_2$. \\ Moreover, the tensor $a$ is symmetric, i.e., $a_{ijkl}= a_{jikl} = a_{klji}$. Also it is positive definite and satisfies
\begin{equation}\label{QPD}
\exists c_0>0,\;\hbox{such that}\; \quad c_0\,S_{ij}S_{ij}\leq a_{ijkl}(y)S_{ij}S_{kl}\quad \hbox{for a.e. $y\in \Yc^*$ and for all $S\in \GS_3$}.
\end{equation}
Note, that the energy density 
$$\widehat{W}_\e(x, \nabla v(x)) = 
\left\{\begin{aligned}
&Q\Big(\frac{x}{\e},E(v)(x)\Big) &&\hbox{if} \quad \det(\nabla v(x))>0,\\
&+\infty  &&\hbox{if} \quad \det(\nabla v(x))\le0,
\end{aligned}\right.\qquad \hbox{for a.e. }x\in \O^*_\e
$$
depends on the strain tensor $\ds E(v)={1\over 2}\big((\nabla v)^T\nabla v-\GI_3\big)$ with $\GI_3$ the unit $3\times 3$ matrix.

\begin{rem}
	As a classical example of a local elastic energy satisfying the above assumptions, we mention the following St Venant-Kirchhoff's law  for which
	$$
	\widehat{W}(F)=\left\{
	\begin{aligned}
	&{\lambda\over 8}\big(tr(F^TF-\GI_3) \big)^2+{\mu\over 4}tr\big((F^TF-\GI_3)^2\big)\qquad&&\hbox{if}\quad \det(F)>0\\
	&+\infty &&\hbox{if}\quad\det(F)\le 0.
	\end{aligned}
	\right.
	$$ 
\end{rem}

Now we are in the position to state the problem. \\[1mm]
Therefore, set
\begin{equation*}
m_{\e}=\inf_{v\in \GV_{\e}} J_{\e}(v)\footnotemark.
\end{equation*} \footnotetext{It is well known that the existence of a minimizer for ${\cal J}_{\e}$ is still an open problem.}

\section{Preliminary estimates}\label{sec:prel-est}
\subsection{Recalls about the plate deformations}
Denote $x'=(x_1,x_2)\in \R^2$ and
$$\GU_{\e}\doteq\Big\{u\in H^1(\O_{\e})^3\;\;|\; \; \; u=0\;\; \hbox{a.e. on } \Gamma_\e\Big\}.$$
The deformations and the terms of their decompositions are estimated in terms of $\|\hbox{dist}(\nabla v, SO(3))\|_{L^2(\O^*_{\e})}$, this is why the lemma below plays a crucial role
\begin{lemma}\label{lem42}
	Let $v\in \GV_\e$ be a deformation and  $\wt{v}\in \GD_\e$ the extended deformation given by Proposition \ref{P1} and Remark \ref{Rem1}.  The  associated displacement $u = \wt{v} - I_d$ belongs to $\GU_\e$ and  satisfies
	\begin{equation}\label{EQ32}
	\|e(u)\|_{L^2(\O_{\e})}  \leq C_0\|dist(\nabla v, SO(3))\|_{L^2(\O^*_\e)} + \frac{C_1}{\e^{5/2}} \|dist(\nabla v, SO(3))\|^2_{L^2(\O^*_\e)} 
	\end{equation}
	The constants do not depend on $\e$ and $v$ (they depend on $\o$, $Y^*$ and $\kappa$).
\end{lemma}
\begin{proof} In \cite[Lemma 4.3]{Shell1} it is proved that there exists a constant which does not depend on $\e$ and $\wt{v}$ such that
$$\|e(u)\|_{L^2(\O_{\e})}  \leq C\|dist(\nabla \wt{v}, SO(3))\|_{L^2(\O_\e)} \Big(1+ \frac{1}{\e^{5/2}} \|dist(\nabla \wt{v}, SO(3))\|_{L^2(\O_\e)}\Big).$$ 
Then, Proposition \ref{P1} gives a constant which does not depend on $\e$ and $v$ such that
$$
\big\|\hbox{dist}\big(\wt{v}, SO(3)\big)\big\|_{L^2(\O_{\e})}\leq C\big\|\hbox{dist}(v, SO(3))\big\|_{L^2(\O^*_{\e})}.
$$
This ends the proof of the lemma.
\end{proof}
	
\subsection{Recalls about the plate displacements}
Set
$$
\begin{aligned}
&H^1_\gamma(\o)\doteq\big\{\phi\in H^1(\o)\;|\; \phi=0\;\; \hbox{a.e. on }\; \gamma\big\},\\
&H^2_\gamma(\o)\doteq\big\{\phi\in H^1(\o)\;|\; \phi=0,\;\;\nabla\phi=0\; \hbox{a.e. on }\; \gamma\big\}.
\end{aligned}
$$
Below we recall a definition from \cite[Chapter 11]{CDG} (see also \cite{GHJ,Gstructureplates}) 
\begin{definition}\label{DEFElem} 
Elementary displacement   are elements $u_e$  of $H^1(\O_{\e})^3$ satisfying   for a.e. $x =(x',x_3)\in  \O_\e$ (where $x'\in \o$)
\begin{equation*}
\begin{aligned}
& u_{e,1}(x)={\cal U}_1(x')+x_3{\cal R}_1(x'),\\
& u_{e,2}(x)={\cal U}_2(x')+x_3{\cal R}_2(x'),\\
& u_{e,3}(x)={\cal U}_3(x').  
\end{aligned}
\end{equation*}
Here 
 $${\cal U}= ({\cal U}_{1}, {\cal U}_{2}, {\cal U}_{3}) \in H^1(\o)^3  \quad \text {and }\quad  {\cal R}={\cal R}_1\Ge_1+{\cal R}_2\Ge_2\in H^1(\o)^2.$$
 \end{definition}

The following lemma is proved in \cite[Theorem 11.4 and Proposition 11.6]{CDG}
\begin{lemma}\label{lem41} Let $u$ be in $\GU_{\e}$.  The displacement $u$ can be  decomposed  as the sum 
\begin{equation} \label{eq 13.1bis}
u= u_{e}+\overline u
\end{equation}
 of an elementary displacement $u_{e}$ and a residual displacement  $\overline u$, both belonging to $\GU_\e$ and satisfying
\begin{equation}\label{displacement-est1}
{\cal U}\in H^1_\gamma(\o)^3,\quad {\cal R}\in H^1_\gamma(\o)^2,\quad\|\overline{u}\|_{L^2(\O_\e)}+ \e\|\nabla \overline{u}\|_{L^2(\O_\e)}  \leq  C\e \|e(u)\|_{L^2(\O_\e)}.
\end{equation}
Moreover, one has
\begin{equation}\label{EQ10-11-NEW}
\begin{aligned}
&\;\|{\cal U}_\alpha\|_{H^1(\omega)}+\e\big(\|{\cal U}_3\|_{H^1(\o)}+\|{\cal R}\|_{H^1(\omega)}\big)\leq \frac{C}{\e^{1/2}}\|e(u)\|_{L^2(\O_{\e})},\\
&\big\|\partial_\alpha \Uc_3 + \Rc_\alpha\big\|_{L^2(\o)}\leq  \frac{C}{\e^{1/2}}\|e(u)\|_{L^2(\O_\e)},\\
& \;\|u_\alpha\|_{L^2(\O_{\e} )}+\e\|u_3\|_{L^2(\O_{\e} )}\le  C\|e(u)\|_{L^2(\O_{\e})},\\
&\sum_{\alpha,\beta=1}^2\Big\|{\partial u_\alpha\over \partial x_\beta}\Big\|_{L^2(\O_\e)}+\Big\|{\partial u_3\over \partial x_3}\Big\|_{L^2(\O_\e)}\le  C\|e(u)\|_{L^2(\O_{\e})},\\
&\sum_{\alpha=1}^2\Big(\Big\|{\partial u_\alpha\over \partial x_3}\Big\|_{L^2(\O_\e)}+\Big\|{\partial u_3\over \partial x_\alpha}\Big\|_{L^2(\O_\e)}\Big)\le  {C\over \e}\|e(u)\|_{L^2(\O_{\e})}.
\end{aligned}
\end{equation} 
	The constants do not depend on $\e$. 
\end{lemma}

\subsection{Assumptions on the forces}\label{SS53}

The forces have to admit a certain scaling with respect to the $\e$-scaling of the domain. For the textile we require forces of the type\\
\begin{align}\label{force}
\begin{split}
f_{\e,1} &= \e^2 f_1,\\
f_{\e,2} &= \e^2 f_2,\\
f_{\e,3} &= \e^3 f_3,
\end{split}
\qquad  \;\hbox{a.e. in }\; \O^*_{\e},
\end{align}
with $f\in L^2(\o)^3$. In order to obtain at the limit a von-K\'{a}rm\'{a}n model, the applied forces must satisfy a condition
\begin{equation}\label{C2}
\|f\|_{L^2(\o)}\leq C^*.
\end{equation} 
This constant depends on the reference cell $\Yc^*$, the mid-surface $\o$ of the structure and the local elastic energy $W$ (see Lemma \ref{lem43}).

The scaling of the force gives rise to the order of the energy in the elasticity problem. We prove this,  in lemma below.

\begin{lemma}\label{lem43}
	Let $v\in \GV_{\e}$ be a deformation such that $J_{\e}(v) \leq 0$. Assume \eqref{force} on the forces. There exists a constant $C^*$ independent of $\e$ and the applied forces such that, if $\|f\|_{L^2(\o)} < C^*$ one has
	\begin{align*}
	\|dist(\nabla v, SO(3))\|_{L^2(\O^*_\e)} \leq C \e^{5/2} \|f\|_{L^2(\O)}.
	\end{align*}
	The constant $C$ does not depend on $\e$.
\end{lemma}
\begin{proof} Using \eqref{QPD}  gives rise to the estimation
	\begin{equation}\label{EQ411}
		c_0 \|dist(\nabla v, SO(3))\|^2_{L^2(\O^*_{\e})} \leq \Big|\int_{\O^*_\e} f_\e \cdot (v - I_d)\; dx\Big|.
	\end{equation}
Introduce $u=\wt{v}-I_d\in \GU_\e$ the associated displacement to the extended deformation (see Lemma \ref{lem42}). Then, with \eqref{force} and the estimates \eqref{EQ10-11-NEW}$_3$  we obtain
\begin{equation}\label{EQ47}
	\begin{aligned}
		\Big|\int_{\O^*_\e} f_\e \cdot (v - I_d)\; dx\Big| &\leq \e^{5/2} \|f_\alpha \|_{L^2(\o)} \|u_\alpha\|_{L^2(\O^*_\e)} + \e^{7/2} \|f_3 \|_{L^2(\o)} \|u_3\|_{L^2(\O^*_\e)}\\
		&\leq \e^{5/2} \|f_\alpha \|_{L^2(\o)} \|u_\alpha\|_{L^2(\O_\e)} + \e^{7/2} \|f_3 \|_{L^2(\o)} \|u_3\|_{L^2(\O_\e)}\\
		&\leq C_2 \e^{5/2} \|f \|_{L^2(\o)} \|e(u)\|_{L^2(\O_\e)}.
	\end{aligned}
	\end{equation}
	Eventually, the above inequality with \eqref{EQ411} and Lemma \ref{lem42} give
\begin{multline*}
c_0\|dist(\nabla v, SO(3))\|^2_{L^2(\O^*_{\e})} \leq C_2C_0\e^{5/2} \|f\|_{L^2(\O)} \|dist(\nabla v, SO(3))\|_{L^2(\O^*_{\e})}\\ + C_2C_1\|f\|_{L^2(\O)} \|dist(\nabla v, SO(3))\|^2_{L^2(\O^*_{\e})}.
\end{multline*}
If $C_2C_1\|f\|_{L^2(\o)} < c_0$, then 
		\begin{align*}
		\|dist(\nabla v, SO(3))\|_{L^2(\O^*_\e)} \leq  \frac{C_2 \e^{5/2}}{c_0- C_2C_1\|f\|_{L^2(\O)} }\|f\|_{L^2(\O)}.
	\end{align*}
	\vskip-7mm
\end{proof}

Now, if the deformation $v\in\GV_\e$ satisfies ${\cal J}_{\e}(v) \leq 0$, one can give a lower bound of  the infimum of the functional ${\cal J}_\e$. To do this, use the assumptions \eqref{QPD} on the problem and \eqref{force}-\eqref{C2} on the forces together with the Lemmas \ref{lem42}-\ref{lem43} and inequality \eqref{EQ47} lead to 
	\begin{align}\label{vvtest}
	c_0\|(\nabla v)^T\nabla v - \GI_3\|^2_{L^2(\O^*_\e)} \leq \int_{\O_\e^*} \wh{W}(\nabla v) dx \leq \int_{\O^*_\e} f_\e \cdot (v - I_d)\; dx \leq C \e^5\|f\|^2_{L^2(\o)}.
	\end{align}
As a consequence, there exists   a constant $c$ independent of $\e$ such that 	
	\begin{align}\label{eq51}
	-c\e^5 \leq {\cal J}_{\e}(v) \leq 0
	\end{align}
	Recalling that $m_{\e}=\inf_{v\in \GV_{\e}} J_{\e}(v)$ yields
	\begin{align}\label{eq52}
	-c\leq {m_{\e}\over \e^5} \leq 0.
	\end{align}
	Our aim is to give the asymptotic behavior  of the rescaled sequence $\ds \Big\{{m_{\e}\over \e^5}\Big\}_\e$ and to characterize its limit as the minimum of a functional.

\section {Asymptotic behavior}
In this section, we consider a sequence $\{v_\e\}_\e$ of deformations satisfying 
\begin{align}\label{assump-dist-5}
	 \ds \|dist(\nabla  v_\e,SO(3))\|_{L^2(\O^*_{\e})} \le C \e^{5/2}.
\end{align} Below we are interested by the asymptotic behavior of  the sequence of displacements $\{u_\e\}_\e=\{v_\e-I_d\}_\e$. Here, for every $\e$, $u_\e$ is the associated displacement to the extended deformation $v_\e$. 

From Lemma \ref{lem42}, one has
\begin{equation}\label{EQ53}
\|e(u_\e)\|_{L^2(\O_{\e})}  \leq C \e^{5/2}.
\end{equation}
Below, we	recall some estimates of Lemma \ref{lem41} under this assumption. One has
		\begin{equation}\label{est_final1}
		\begin{aligned}
		&\|\overline{u}_\e\|_{L^2(\O_\e)} + \e\|\nabla \overline{u}_\e\|_{L^2(\O_\e)} \leq C \e^{7/2},\\
		&\big\|{\cal U}_{\e,3}\big\|_{H^1(\o)}+\big\|{\cal R}_\e\big\|_{H^1(\o)}\le C\e ,\qquad\big\|{\cal U}_{\e,\alpha}\big\|_{H^1(\o)}\le C\e^2,\\ 
		&\big\|\partial_\alpha \Uc_{\e,3} + \Rc_{\e,\alpha}\big\|_{L^2(\o)}\le C \e^2.
		\end{aligned}
		\end{equation}
The constants do not depend on $\e$.

\begin{lemma}[See {\cite[Section 7]{Shell1}}]\label{lemconv1}
Under the assumptions of Lemma \ref{lem41}, there exist a subsequence of $\{\e\}$, still denoted $\{\e\}$,   $\Uc_3 \in H^2_\gamma(\o)$ and $\Uc_\alpha,\; \Rc_\alpha \in H^1_\gamma(\o)$ ($\alpha\in\{1,2\}$) such that
	\begin{align}\label{eq48}
	\begin{split}
	&\frac{1}{\e}\Uc_{\e,3} \rightarrow \Uc_3\quad \text{strongly in } H^1_\gamma(\o),\\
	&\frac{1}{\e}\Rc_{\e,\alpha} \rightharpoonup \Rc_\alpha\quad \text{weakly in } H^1_\gamma(\o),\;\; \hbox{and strongly in } L^4(\o)\\
	&\frac{1}{\e^2}\Uc_{\e,\alpha} \rightharpoonup \Uc_\alpha \quad \text{weakly in } H^1_\gamma(\o),\\
	&\frac{1}{\e^2} \big(\partial_\alpha \Uc_{\e,3} + \Rc_{\e,\alpha}\big) \rightharpoonup \Zc_\alpha \quad \text{weakly in } L^2(\o).
	\end{split}
	\end{align}
	Moreover, one has
	\begin{align}\label{eq46}
	\partial_\alpha \Uc_3 +\Rc_\alpha=0.
	\end{align}
\end{lemma}
\begin{proof}
	The convergences and equalities are easy consequences of the above estimates \eqref{est_final1}. To see the strong convergence of \eqref{eq48}$_1$ note that  \eqref{EQ10-11-NEW}$_{1,2}$ and the strong convergence of $\ds \frac{1}{\e}\Rc_{\e,\alpha}$ in $L^2(\O)$ imply
	\begin{align*}
		\frac{1}{\e}\partial_\alpha \Uc_{\e,3} = \frac{1}{\e}(\partial_\alpha  \Uc_{\e,3} + \Rc_{\e,\alpha}) - \frac{1}{\e}\Rc_{\e,\alpha} \rightarrow 0 - \Rc_\alpha = \partial_\alpha\Uc_3,\quad \text{strongly in }L^2(\o).
	\end{align*}
	the last equality comes from \eqref{eq46}.
\end{proof}

\subsection{The unfolding and the unfolding and rescaling operators}

For the asymptotic behavior we introduce two operators: $\Te$ for the homogenization in $\o$ and $\mathfrak{T}_\e$ for the homogenization and dimension reduction in $\O^*_\e$. Both operators can be found in \cite{CDG} thus we recall here only the important properties.
\begin{definition} For every measurable function $\phi \in L^1(\o)$ we recall the definition of the unfolded function $\Te(\phi) \in L^1(\o\times \Yc')$ 
$$\Te(\phi)(x',y') = \phi\Big(\e \Big[\frac{x'}{\e}\Big]+ \e y'\Big)\qquad \hbox{for a.e. }(x',y')\in \o\times \Yc'.$$
For every measurable function $\psi \in L^1(\O^*_\e)$	the unfolding and the rescaling operator $\mathfrak{T}_\e$ is defined by
$$\mathfrak{T}_\e(\psi)(x', y) = \psi\Big(\e \Big[\frac{x'}{\e}\Big]+ \e y\Big)\qquad \hbox{for a.e. }(x',y)\in \o\times \Yc^*.$$
\end{definition}
 
\begin{lemma}\label{lemconv2}
	There exist a subsequence of $\{\e\}$, still denoted $\{\e\}$,  $\wh{\Uc}_\alpha, \wh{\Rc}_\alpha \in L^2(\o ; H^1_{per}(\Yc'))$ and $\mathfrak{u} \in L^2(\o ; H^1_{per}(\Yc'))$ such that 
	\begin{align}\label{eq415}
	\begin{split}
	&\frac{1}{\e}\Te(\nabla \Uc_{\e,3}) \longrightarrow \nabla \Uc_3\qquad \hbox{ strongly in } L^2(\o\times \Yc')^2\\
	&\frac{1}{\e}\Te(\Rc_{\e}) \longrightarrow \Rc\qquad \hbox{ strongly in } L^2(\o\times \Yc')^2\\
	&\frac{1}{\e^2}\Te(\nabla \Rc_{\e,\alpha}) \rightharpoonup \nabla \Rc_\alpha + \nabla_{y} \widehat{\Rc}_\alpha\qquad \hbox{ weakly in } L^2(\o\times \Yc')^2\\
	&\frac{1}{\e^2}\Te(\nabla \Uc_{\e,\alpha}) \rightharpoonup \nabla \Uc_\alpha + \nabla_{y} \widehat{\Uc}_\alpha\qquad \hbox{ weakly in } L^2(\o\times \Yc')^2\\
	& \frac{1}{\e^2}\Te(\partial_\alpha\Uc_{\e,3} +\Rc_{\e,\alpha}\big) \rightharpoonup \Zc_\alpha + \nabla_{y_\alpha} \mathfrak{u} +  \wh{\Rc}\qquad \hbox{ weakly in } L^2(\o\times \Yc').
	\end{split}
	\end{align}
Moreover, there exists $\overline{u}\in L^2(\o ; H^1_{per}(\Yc^*))^3$ such that
	\begin{align}\label{EQ57}
		\begin{split}
		&\frac{1}{\e^2} \TRe(\overline{u}_\e) \rightharpoonup \overline{u} \qquad \text{weakly in } L^2(\o ; H^1(\Yc^*)),\\
		&\frac{1}{\e^2} \TRe(\nabla \overline{u}_\e) \rightharpoonup \nabla_y \overline{u} \qquad \text{weakly in } L^2(\o\times \Yc^*)^9.
		\end{split}
	\end{align}	
 Furthermore, one has
\begin{equation}\label{EQ59}
{1\over \e^2}\TRe\big(\nabla u_\e(\nabla u_\e)^T\big)\longrightarrow \begin{pmatrix}
\partial_1\Uc_3\partial_1\Uc_3 & \partial_1\Uc_3\partial_2\Uc_3 & 0\\
\partial_1\Uc_3\partial_2\Uc_3 & \partial_2\Uc_3\partial_2\Uc_3 & 0\\
0 & 0 & \nabla \Uc_3\cdot\nabla \Uc_3
\end{pmatrix}\quad \hbox{strongly in } L^1(\o\times \Yc^*)^{9}.
\end{equation} The above convergence is weak in 
$L^2(\o\times \Yc^*)^{9}$.\\[1mm]
\end{lemma}
\begin{proof} The first convergence \eqref{eq415}$_{1}$ is a consequence of \eqref{eq48} and the classical results of the PUM.
	Convergences \eqref{eq415}$_{2,3,4}$ come from the convergences in Lemma \ref{lemconv1} and again of the classical results of the PUM (see \cite{CDG}).
	The last convergence \eqref{eq415}$_{5}$ is a consequence of \cite[Lemma 11.11]{CDG}, together with the convergences \eqref{eq48}$_4$ and \eqref{eq415}$_{3}$.\\[1mm]
		\begin{align*}
			 \nabla u_\e 
			= \begin{pmatrix}
			\partial_1 \Uc_{\e,1} + x_3 \partial_1 \Rc_{\e,1} & \partial_2 \Uc_{\e,1} + x_3 \partial_2 \Rc_{\e,1}& \Rc_{\e,1}\\
			\partial_1 \Uc_{\e,2} + x_3 \partial_1 \Rc_{\e,2}  & \partial_2 \Uc_{\e,2} + x_3 \partial_2 \Rc_{\e,2}& \Rc_{\e,2}\\
			\partial_1 \Uc_{\e,3} & \partial_2 \Uc_{\e,3}  & 0
			\end{pmatrix}+\nabla \overline{u}_\e
		\end{align*}

		Then since \eqref{eq415}$_{1,2}$ are strong convergences and the other fields converge to zero due to \eqref{eq415}$_{3,2}$ and \eqref{EQ59} we obtain
		\begin{align}\label{ConvN}
			{1\over \e}\TRe\big(\nabla u_\e\big)\longrightarrow  \begin{pmatrix}
			0 & 0 & \Rc_1\\
			0 & 0 & \Rc_2\\
			\partial_1\Uc_3 & \partial_2\Uc_3 & 0\\
			\end{pmatrix}\quad \hbox{strongly in } L^2(\o\times \Yc^*)^{9}.
		\end{align} 
Hence, using \eqref{eq46} this yields
			\begin{multline*}
			{1\over \e^2}\TRe\big(\nabla u_\e(\nabla u_\e)^T\big)\longrightarrow \begin{pmatrix}
			0 & 0 & \Rc_1\\
			0 & 0 & \Rc_2\\
			-\Rc_1 & -\Rc_2 & 0\\
			\end{pmatrix}\begin{pmatrix}
			0 & 0 & -\Rc_1\\
			0 & 0 & -\Rc_2\\
			\Rc_1 & \Rc_2 & 0\\
			\end{pmatrix}\\
			= \begin{pmatrix}
			\Rc_1\Rc_1& \Rc_1\Rc_2 & 0\\
			\Rc_1\Rc_2 &\Rc_2\Rc_2 & 0\\
			0 & 0 & \Rc_1^2 + \Rc_2^2
			\end{pmatrix}\quad \hbox{strongly in } L^1(\o\times \Yc^*)^{9}.
			\end{multline*}
Now note that the sequence is actually bounded in $L^2(\O \times \Yc^*)$ by
	\begin{align*}
		\|\nabla u_\e(\nabla u_\e)^T\|_{L^2(\O_\e)} &\leq \|\nabla u_\e(\nabla u_\e)^T+ 2 e(u_\e)\|_{L^2(\O_\e)} + 2 \|e(u_\e)\|_{L^2(\O_\e)}\\
		&=  \|\nabla v_\e(\nabla v_\e)^T-\GI_3\|_{L^2(\O_\e)} + 2 \|e(u_\e)\|_{L^2(\O_\e)} \leq C \e^{5/2}
	\end{align*}
Hence, the sequence $\ds\Big\{{1\over \e^2}\TRe\big(\nabla u_\e(\nabla u_\e)^T\big)\Big\}_\e$ is bounded in $L^2(\O\times \Yc^*)^9$.	This ensures that \eqref{EQ59} is also weakly converging in 
$L^2(\O\times \Yc^*)$.
\end{proof}

Eventually, we find with \eqref{ConvN} 
\begin{align}\label{defo_conv}
	\mathfrak{T}_\e(\nabla v_\e) \rightarrow \GI_3 \qquad \text{strongly in } L^2(\o\times \Yc^*)^9.
\end{align}
Additionally, the displacements converge as follows 
\begin{align}\label{u-conv}
\begin{split}
	&\frac{1}{\e^{2}} \mathfrak{T}_\e (u_{\e,\alpha}) \rightarrow \Uc_\alpha - y_3 \partial_\alpha \Uc_3 \qquad \text{strongly in } L^2(\O\times \Yc),\\
	&\frac{1}{\e^{1}} \mathfrak{T}_\e (u_{\e,3}) \rightarrow \Uc_3 \qquad \text{strongly in } L^2(\O\times \Yc).
\end{split}
\end{align}
The above convergences show that the limit displacement is of Kirchhoff-Love type.\\
Then, we have
\begin{lemma}\label{lemconve}
\noindent For a subsequence  we have
\begin{equation}\label{EQ510}
{1\over 2\e^{2}}\TRe \big((\nabla v_\e)^T\nabla  v_\e -\GI_3\big)\rightharpoonup 
\GE(\Uc) + e_{y}(\widehat{u})\qquad\hbox{weakly in}\qquad L^2(\o \times \Yc^* )^9,
\end{equation}
 where the symmetric matrix $\GE(\Uc)$ is defined by
\begin{align*}
\GE(\Uc)=\begin{pmatrix}
-y_3 {\partial^2 {\cal U}_3\over \partial x_1^2}+{\cal Z}_{11} & -y_3 {\partial^2 {\cal U}\over \partial x_1\partial x_2}+{\cal Z}_{12} 
& 0 \\
* & -y_3 {\partial^2 {\cal U}_3\over \partial x_2^2}+{\cal Z}_{22}  & 0\\
0 & 0&   0
\end{pmatrix}
\end{align*}
where
\begin{multline*}
\widehat{u}(x',y) =\\
 \overline{u}(x',y) + \frac{y_3}{2}(\Zc_1(x')\cdot\Ge_3)\Ge_1 + \frac{y_3}{2}(\Zc_2(x')\cdot\Ge_3)\Ge_2 + y_3\widehat{\Rc}(x',y')\land \Ge_3 + \mathfrak{u}(x',y')+y_3|\nabla \Uc_3(x')|^2\Ge_3\\
\hbox{for a.e. } (x',y)\in \o\times \Yc^*
\end{multline*}
and 
$${\cal Z}_{\alpha\beta}= e_{\alpha\beta}({\cal U})+{1\over 2}{\partial {\cal U}_3\over \partial x_\alpha}{\partial {\cal U}_3\over \partial x_\beta}\qquad (\alpha,\beta)\in\{1,2\}^2.$$
\end{lemma}
\begin{proof}
First, in the strain tensor $\nabla v (\nabla v)^T -\GI_3$ replace  the deformation by its associated displacement $u = v- I_d$. This yields
\begin{align}
	\nabla v (\nabla v)^T -\GI_3 = \nabla u (\nabla u)^T  + \nabla u + (\nabla u)^T = \nabla u (\nabla u)^T  + 2 e(u).
\end{align}
The first term on the right hand side is already covered in \eqref{EQ59}. Hence, consider now $\frac{1}{\e^2} \TRe (e(u))$. However, this is already done in \cite{CDG} and yields
\begin{align}\label{linstrainlimit}
	\frac{1}{\e^2} \TRe (e(u)) \rightharpoonup 
	\begin{pmatrix}
	e_{11}(\Uc) -y_3 \frac{\partial^2\Uc_3}{\partial x_1^2} & e_{12}(\Uc) -y_3 \frac{\partial^2\Uc_3}{\partial x_1\partial_2} & 0\\
	* & e_{22}(\Uc) -y_3 \frac{\partial^2\Uc_3}{\partial x_2^2} & 0\\
	0 & 0 & 0
	\end{pmatrix} 
	+ e_y(\wh{u}), \quad \text{weakly in } L^2(\O \times \Yc^*),
\end{align}
where we define
\begin{align*}
	\widehat{u}(x',y) = \overline{u}(x',y) + \frac{y_3}{2}(\Zc_1(x')\cdot\Ge_3)\Ge_1 + \frac{y_3}{2}(\Zc_2(x')\cdot\Ge_3)\Ge_2 + y_3\widehat{\Rc}(x',y')\land \Ge_3 + \mathfrak{u}(x',y') +y_3|\nabla \Uc_3(x')|^2\Ge_3
\end{align*}
 for a.e. $(x',y)\in \o\times \Yc^*$. Upon rewriting the result this yields the claim.
\end{proof}

Note, that the antisymmetric part is responsible for the nonlinearity of the problem.
Finally we prove that in the limit problems and in the case of glued yarns, one can replace the $e_{\alpha\beta}({\cal U})$'s with the ${\cal Z}_{\alpha\beta}({\cal U})$'s.

\subsection{The limit problem}

The limits of the previous section allow to investigate the limit of the elastic problem. Therefore recall the energy of the elasticity problem in the limit 
\begin{align} \label{limenergy}
{\cal J}(\Uc,\widehat{u}) = \int_\o \int_{\Yc^*} \widehat{W}\big(y, \GE(\Uc) + e_{y}(\widehat{u})\big)\; dydx' -|\Yc^*| \int_\o f \cdot \Uc \; dx'.
\end{align}
Define the limit space
\begin{align*}
\UU:= \left\{ \Uc=\big(\Uc_1,\Uc_2,\Uc_3\big) \in  H^1(\o)^2 \times H^2(\o)\; \mid\; \Uc=0,\;\;  \partial_\alpha \Uc_3 = 0\quad \hbox{a.e.  on } \gamma \right\}
\end{align*}
Furthermore, set 
\begin{align}
J(\Uc,\wh{u}) &=  \int_\o \int_{\Yc^*} \widehat{W}\big(y, \GE(\Uc) + e_{y}(\widehat{u})\big)\; dydx',
\end{align}
the part of the energy without the external force. Thus we can write
\begin{align*}
{\cal J}(\Uc,\wh{u}) &= J(\Uc,\wh{u}) -|\Yc^*| \int_\o f \cdot \Uc \; dx'.
\end{align*}
First we prove that the functional ${\cal J}$ admits a minimum on $\UU \times L^2(\o ; H^1_{per}(\Yc^*))^3$.\\
For every $(\xi,\zeta,\wh{w})\in \mathbb{S}\doteq\mathbb{R}^3\times\mathbb{R}^3\times H_{per,0}^1(\Yc^*)^3$ denote $\wt{\cal E}$ the symmetric matrix
$$
\wt{\cal E}(\xi,\zeta,\wh{w})=
\begin{pmatrix}
\xi_1-y_3\zeta_1+ e_{11,y}(\wh w)  & \xi_3 -y_3 \zeta_3 + e_{12,y}(\wh w) 
&  e_{13,y}(\wh w)  \\
* & \xi_2-y_3\zeta_2 + e_{22,y}(\wh w)  &  e_{23,y}(\wh w) \\
* & *&    e_{33,y}(\wh w) 
\end{pmatrix}
$$

\begin{lemma}\label{lem_strain_semi}
We equip the space $\mathbb{S}\doteq\mathbb{R}^3\times\mathbb{R}^3\times H_{per,0}^1(\Yc^*)^3$ with the semi-norm
	$$	\|(\xi,\zeta,\wh{w})\|_\mathbb{S} = \sqrt{\sum_{i,j=1}^3  \|\wt{\cal E}_{ij}(\xi,\zeta,\wh{w})\|^2_{L^2( \Yc^*)}}.
	$$   
	Then, this expression actually defines a norm on $\mathbb{S}$ equivalent to the product-norm.
\end{lemma}
\begin{proof} To show that the semi-norm is actually a norm it is necessary to show the positive definiteness, i.e., $\|(\xi,\zeta,\wh{w})\|_\mathbb{S} = 0$ implies $(\xi,\zeta,\wh{w})=0$.
	
	Let $(\xi,\zeta,\wh{w}) \in \mathbb{S}$ satisfy $\|(\xi,\zeta,\wh{w})\|_\mathbb{S} = 0$ and define the map
	\begin{align*}
		\tau(y) = \begin{pmatrix}
		y_1\left(\xi_1 - y_3\zeta_1\right) + y_2\left(\xi_3 - y_3\zeta_3\right)\\
		y_1\left(\xi_3 - y_3\zeta_3\right) + y_2\left(\xi_2 - y_3\zeta_2\right)\\
		-\frac{y_1^2}{2}\zeta_1-\frac{y_2^2}{2}\zeta_2 - y_1y_2\zeta_3
		\end{pmatrix}.
	\end{align*}
	Then rewrite
	\begin{align}
		\wt{\cal E}(\xi,\zeta,\wh{w}) = e_{y}(\tau + \wh w).
	\end{align}
	Hence, $\tau(y) + \wh w(y) = a + b\land y$ is a rigid motion. Then the properties of $\wh{w}\in H^1_{per,0}(\Yc^*)$ (periodicity in the directions $\Ge_1$, $\Ge_2$ and vanishing mean) imply 
	that $a=b=\xi=\zeta=0$ and thus also $\wh{w}=0$.
	
	Finally, by a contradiction argument it is easy to prove that there exists a constant $C$ such that
	\begin{align}
		\|\xi\|_2 + \|\zeta\|_2 + \|\wh{w}\|_{H^1_{per,0}(\Yc^*)} \leq C \|(\xi,\zeta,\wh{w})\|_\mathbb{S}
	\end{align}
	holds for all $(\xi,\zeta,\wh{w}) \in \mathbb{S}$.
\end{proof}

\begin{lemma}
The functional ${\cal J}$ admits a minimum on $\UU \times L^2(\o ; H^1_{per,0}(\Yc^*))^3$.
\end{lemma}
\begin{proof} First, from \eqref{QPD} and Lemma \ref{lem_strain_semi}, there exists a constant $C>0$ such that
\begin{equation}\label{EQ513}
\begin{aligned}
C\Big(\sum_{\alpha,\beta=1}^2\Big[&\Big\|e_{\alpha\beta}(\Wc)+{1\over 2}{\partial \Wc_3\over \partial x_\alpha}{\partial {\cal W}_3\over \partial 
x_\beta}\Big\|^2_{L^2(\o)}+\Big\|{\partial^2\Wc_3\over \partial x_\alpha\partial x_\beta}\Big\|^2_{L^2(\o)}\Big]+\|\wh{w}\|^2_{L^2(\o; H^1(\Yc^*))}\big)\leq J(\Wc,\wh{w}),\\
& \hbox{for all } (\Wc,\wh{w})\in \UU \times L^2(\o ; H^1_{per,0}(\Yc^*))^3.
\end{aligned}
\end{equation}
Set
			\begin{align*}
			m = \inf_{(\Uc,\wh{u})\in \UU \times L^2(\o ; H^1_{per,0}(\Yc^*))^3} {\cal J}(\Uc,\wh{u})
			\end{align*}
			where $m\in [-\infty, 0]$.\\[1mm]
\noindent {\it Step 1.} We show that 	$m\in (-\infty, 0]$.\\[1mm]		
To show that $m$ is actually finite we show that the sequence is bounded and thus admits a weak convergent subsequence and then use the weak sequential continuity of ${\cal J}$.
			
			The boundedness of $\Uc_i$ are show with the help of the functional $J$.
			Now, consider first $\Uc_3$, which using \eqref{EQ513} together with the boundary conditions satisfies
			\begin{align}\label{U3est}
\|\wh{u}\|^2_{L^2(\o; H^1(\Yc^*))}\leq J(\Uc,\wh{u}),\qquad 	\|\Uc_3\|^2_{H^2(\o)} \leq C\sum_{\alpha,\beta=1}^2\Big\|{\partial^2\Uc_3\over \partial x_\alpha\partial x_\beta}\Big\|^2_{L^2(\o)} 
\leq C_2 J(\Uc,\wh{u}).
			\end{align}
			Similarly, the estimate for $\Uc_\alpha$ is obtained. For this keep in mind that in the energy only $\Zc_{\alpha\beta}$ arise and we arrive at
			\begin{align*}
			\sum_{\alpha,\beta=1}^2 \|e_{\alpha\beta}(\Uc)\|^2_{L^2(\o)} &\leq cJ(\Uc,\wh{u}) + \|\nabla\Uc_3\|^4_{L^4(\o)}\leq cJ(\Uc,\wh{u}) +\|\Uc_3\|^4_{H^2(\o)} \leq cJ(\Uc,\wh{u}) + \left[C_1C_2 J(\Uc,\wh{u})\right]^2.
			\end{align*}
			Note that we used here the embedding $H^2(\o) \hookrightarrow W^{1,4}(\o)$.
			The 2D-Korn inequality then yields 
			\begin{align}\label{U1U2est}
			\|\Uc_1\|^2_{H^1(\O)} + \|\Uc_2\|^2_{H^1(\O)} \leq cJ(\Uc,\wh{u}) + \left[C_1C_2 J(\Uc,\wh{u})\right]^2.
			\end{align}
			
			With \eqref{U3est} and \eqref{U1U2est} we have for the sequence that
			\begin{align}\label{functionalest}
			\begin{split}
			J(\Uc,\wh{u}) &\leq \|f_3\|_{L^2(\o)}\|\Uc_3\|_{L^2(\o)} + \sqrt{\|f_1\|^2_{L^2(\o)}+\|f_2\|^2_{L^2(\o)}}\left[\|\Uc_1\|_{L^2(\o)} + \|\Uc_2\|_{L^2(\o)}\right]\\
			&\leq \|f_3\|_{L^2(\o)} \sqrt{J(\Uc,\wh{u})} + \sqrt{\|f_1\|^2_{L^2(\o)}+\|f_2\|^2_{L^2(\o)}} \left[c\sqrt{J(\Uc,\wh{u})} + C_1C_2 J(\Uc,\wh{u}) \right]
			\end{split}
			\end{align}
			
			Thus we also have
			\begin{align}\label{Jest}
			\begin{split}
			J(\Uc,\wh{u}) &\leq c \|f\|_{L^2(\o)}\sqrt{J(\Uc,\wh{u})} + C_1C_2 \|f\|_{L^2(\o)} J(\Uc,\wh{u})
			\end{split}
			\end{align}
			which shows that $J(\Uc,\wh{u})$ is bounded if and only if $C_1C_2\|f\|_{L^2(\o)}\leq 1$, which is the same constraint as before in Lemma \ref{lem43}. Then, we have
			\begin{multline*}
\forall (\Uc,\wh{u})\in \UU \times L^2(\o ; H^1_{per,0}(\Yc^*))^3,\quad \\
{\cal J}(\Uc,\wh{u})\leq 0\quad \Longrightarrow \quad \|\Uc_1\|_{H^1(\o)}+\|\Uc_2\|_{H^1(\o)}+\|\Uc_3\|_{H^2(\o)}+\|\wh{u}\|_{L^2(\o; H^1(\Yc^*))}\leq C.
			\end{multline*}
Then we easily show that $m\in (-\infty,0]$.\\[1mm]
\noindent{\it Step 2.} We show that $m$ is a minimum.\\[1mm]
			Consider a minimizing sequence $\{(\Uc^n,\wh{u}^n)\}_n \subset \UU \times L^2(\o ; H^1_{per,0}(\Yc^*))^3$ satisfying ${\cal J}(\Uc^n,\wh{u}^n)\leq {\cal J}(0,0) = 0$ and
			\begin{align*}
			m = \inf_{(\Uc,\wh{u})\in \UU} {\cal J}(\Uc,\wh{u}) = \lim_{n\rightarrow+\infty} {\cal J}(\Uc^n,\wh{u}^n).
			\end{align*}
From step 1, one has
$$ 
\|\Uc^n_1\|_{H^1(\o)}+\|\Uc^n_2\|_{H^1(\o)}+\|\Uc^n_3\|_{H^2(\o)}+\|\wh{u}^n\|_{L^2(\o; H^1(\Yc^*))}\leq C
$$ where the constant does not depend on $n$.\\
			Hence, there exists a  subsequence of $\{(\Uc^n,\wh{u}^n)\}_n$, still denoted $\{(\Uc^n,\wh{u}^n)\}_n$,  such that $$(\Uc^n,\wh{u}^n) \rightharpoonup  (\Uc',\wh{u}')\quad\text{weakly 
			in }\UU \times L^2(\o ; H^1_{per,0}(\Yc^*))^3.$$
				Furthermore, by the lower semi-continuity of ${\cal J}$ it is clear that 
				\begin{align}
					{\cal J} (\Uc',\wh{u}') = \liminf_{n\rightarrow+\infty} {\cal J}(\Uc^n,\wh{u}^n) \leq \lim_{n\rightarrow+\infty} {\cal J}(\Uc^n,\wh{u}^n) \leq m
				\end{align}
				However, since $m = \inf_{(\Uc,\wh{u})\in \UU} {\cal J}(\Uc,\wh{u})$ we conclude that for every $(\Uc,\wh{u})\in \UU \times L^2(\o ; H^1_{per,0}(\Yc^*))^3$ it holds
				\begin{align*}
					{\cal J} (\Uc',\wh{u}')\leq m \leq {\cal J}(\Uc,\wh{u}).
				\end{align*}
				This proves that the infimum is in fact a minimum.
\end{proof}
\begin{theorem}
	Under the assumptions on the forces \eqref{force}-\eqref{C2} we have
	\begin{equation}\label{EQ422}
		m=\lim_{\e\rightarrow 0} \frac{m_\e}{\e^5} = \min_{(\Uc,\wh{u}) \in \UU \times L^2(\o ; H^1_{per}(\Yc^*))^3} \mathcal{J}(\Uc,\wh{u}).
	\end{equation}
\end{theorem}
\begin{proof} To show this result, we use a kind of $\Gamma$-convergence technique. \\[1mm]\noindent{\it Step 1.}  In this step we show that
$$
\min_{(\Uc,\wh{u}) \in \UU \times L^2(\o ; H^1_{per}(\Yc^*))^3} \mathcal{J}(\Uc,\wh{u})\leq \liminf_{\e\rightarrow 0} \frac{m_\e}{\e^5}.
$$	
To show this, let $\{v_\e\}_\e$, $v_\e \in V_\e$, be a minimizing sequence of deformations. It satisfies
$$
\lim_{\e\rightarrow 0} \frac{\mathcal{J}_\e( v_\e)}{\e^5} = \liminf_{\e\rightarrow 0} \frac{m_\e}{\e^5}.
$$
	Without lost of generality, we can assume that the sequence satisfies $\mathcal{J}_\e(v_\e) \leq 0$ and hence the estimates of the previous sections yield
	\begin{align}
		\|dist(\nabla v_\e,SO(3))\|^2_{L^2(\O^*_\e)} \leq  C \e^5\qquad \text{and} \qquad \|(\nabla v_\e)^T\nabla v_\e - \GI_3\|^2_{L^2(\O^*_\e)} \leq C\e^5.
	\end{align}
 Therefore, we are allowed to use the decomposition defined in \ref{DEFElem} and yields the estimates \eqref{est_final1} and convergences as in Lemma \ref{lemconve} and \ref{lemconv2}. Then the assumptions on the force lead to
	\begin{align*}
		 \lim_{\e\rightarrow 0} \frac{1}{\e^5}\int_{\o \times \Yc^*} \TRe (f_\e \cdot (v_\e - I_d)) dx' dy 
		 &= \lim_{\e\rightarrow 0} \frac{1}{\e^5}\int_{\o \times \Yc^*} \TRe (f_\e \cdot u_\e) dx' dy\\
		 &= |\Yc^*| \int_{\o} f \cdot \Uc \, dx',
	\end{align*}
	converging as a product of a weak and a strong convergence. As consequence, we have  with the weak convergence of the strain tensor \ref{EQ510} together with the weak lower semi-continuity of ${\cal J}$ that 
	\begin{align}
		\liminf_{\e \rightarrow 0} \frac{\mathcal{J}_\e(v_\e)}{\e^5} \geq J(\Uc, \wh{u}) - |\Yc^*| \int_{\o} f \cdot \Uc \,dx
	\end{align}
\noindent{\it Step 2.} We show that for every $(\Uc^\prime,\wh{u}^\prime)\in \UU \times L^2(\o ; H^1_{per}(\Yc^*))^3$, one has
\begin{equation}\label{EQ516}
 \limsup_{\e \rightarrow 0} \frac{m_\e}{\e^5} \leq  \mathcal{J}(\Uc^\prime,\wh{u}^\prime).
 \end{equation}
To do that, let $(\Uc^\prime,\wh{u}^\prime)$ be in $\UU \times L^2(\o ; H^1_{per}(\Yc^*))^3$. We will build a sequence $\{v_\e\}_\e$ of admissible deformations such that
$$\limsup_{\e \rightarrow 0} \frac{m_\e}{\e^5} \leq \lim_{n\rightarrow+\infty}\lim_{\e\rightarrow 0} \frac{\mathcal{J}_\e(V_\e^{(n)})}{\e^5} = \mathcal{J}(\Uc^\prime,\wh{u}^\prime).$$
Consider a sequence  $\{\Uc^{(n)}\}_n$ in $\VV \cap \big(\mathcal{C}^1(\overline{\o})^2 \times \mathcal{C}^2(\overline{\o})\big)$ and $\{\wh{u}^{(n)}\}_n$ in  $L^2(\O ; H^1_{per}(\Yc^*))^3\cap 
{\cal C}^1(\overline{\o}\times \overline{\Yc^*})^3$,  where we additionally assume that $\widehat{u}^\prime_{n|x_2=0} = 0$, such that	
\begin{equation}\label{EQ528}
\begin{aligned}
		&\Uc^{(n)}_\alpha \rightarrow \Uc^\prime_\alpha \qquad \text{strongly in } H^1(\o)\\
		&\Uc^{(n)}_3\rightarrow \Uc^\prime_3 \qquad \text{strongly in } H^2(\o)\\
		&\widehat{u}^{(n)}\rightarrow\widehat{u}^\prime \qquad \text{strongly in } L^2(\o ; H^1_{per}(\Yc^*)).
\end{aligned} 
\end{equation}
Now, we show that there exists a sequence $\{v_\e\}_\e$ such that 
$$ {\limsup}_{\e \rightarrow 0} \frac{m_\e}{\e^5} \leq  \mathcal{J}(\Uc^{(n)},\wh{u}^{(n)}).$$
We define the sequence of deformations
	\begin{align*}
		&V_{\e,1}^{(n)}(x) = x_1 + \e^2\Big(\Uc_1^{(n)}(x_1,x_2) - \frac{x_3}{\e} \partial_1\Uc_3^{(n)}(x_1,x_2) + \e \wh{u}_1^{(n)}(x_1,x_2,\frac{x_3}{\e}) \Big)\\
		&V_{\e,2}^{(n)}(x) = x_2 + \e^2\Big(\Uc_2^{(n)}(x_1,x_2) - \frac{x_3}{\e} \partial_2\Uc_3^{(n)}(x_1,x_2) + \e \wh{u}_2^{(n)}(x_1,x_2,\frac{x_3}{\e})\Big)\\
		&V_{\e,1}^{(n)}(x) = x_3 + \e\Big(\Uc_3^{(n)}(x_1,x_2) + \e^2 \wh{u}_3^{(n)}(x_1,x_2,\frac{x_3}{\e})\Big)
	\end{align*}
	and by construction we have $V \in \GV_\e$. Obviously, the deformation can be further restricted to the original structure $\O^*_{\e}$.

 	Now we are interested in the convergences of the deformations $\{V^{(n)}_\e\}_\e$.
 	Note that they satisfy
 	\begin{align*}
 		\|\nabla V_\e^{(n)} - \GI_3 \|_{L^\infty(\O_\e)} \leq C(n) \e,
 	\end{align*}
 	which is why we can assume that $\det(\nabla V_\e^{(n)}) > 0$ for all $n\in \N$ and all $x \in \O^*_\e$ (if $\e$ is small enough). This leads us together with the right-hand-side to
 	\begin{align}\label{lowerbound}
		m_\e \leq \mathcal{J}_\e(V_\e^{(n)}).
 	\end{align}
 Since the convergence of the deformation components are known, we obtain
	 \begin{align*}
	 	\frac{1}{2\e^2} \TRe \big(\big(\nabla  V_\e^{(n)} \big)^T\nabla  V_\e^{(n)} - \GI_3 \big) \longrightarrow \mathbf{E}(\Uc^{(n)}) + e_{y}\big(\wh{u}^{(n)}\big)\quad \hbox{strongly in } 
	 	L^2(\o\times \Yc^*)
	 \end{align*}
	 defined as in Lemma \ref{lemconve}.
 This convergence gives rise to the convergence of the elastic energy 
 	 \begin{align*}
\lim_{\e\rightarrow 0}\frac{1}{\e^5} \mathcal{J}_\e(V_\e^{(n)})
&=\lim_{\e\rightarrow 0}\frac{1}{\e^5} \int_{\o \times \Yc^*} \TRe \big(\widehat{W}\big(y,\nabla V^{(n)}_\e\big)\big) dx' dy\\
&=\lim_{\e\rightarrow 0}\frac{1}{\e^5} \int_{\o \times \Yc^*} \TRe \big(Q\big(y,(\nabla V^{(n)}_\e)^T\nabla V^{(n)}_\e - \GI_3\big)\big) dx' dy\\ 
&=\int_{\o \times \Yc^*}  Q\big(y,\mathbf{E}(\Uc^{(n)}) + e_{y}(\wh{u}^{(n)})\big) dx' dy
	 \end{align*}
	 and the right-hand-side
	 \begin{align*}
	 	\lim_{\e\rightarrow 0}\frac{1}{\e^5} \int_{\o \times \Yc^*} \TRe \big( f_\e \cdot  \big(V^{(n)}_\e - I_d\big)\big) dx' dy \longrightarrow  |\Yc^*| \int_{\o}   f_\e \cdot  \Uc^{(n)} dx' dy
	 \end{align*}
	 Hence with \eqref{lowerbound} we obtain
	 \begin{align*}
	 	\mathop{\lim\sup}_{\e\rightarrow 0} \frac{m_\e}{\e^5} \leq \lim_{\e\rightarrow 0}\frac{1}{\e^5} \mathcal{J}_\e(V_\e^{(n)})=\mathcal{J}\big(\Uc^{(n)},\wh{u}^{(n)}\big).
	 \end{align*}
	 Since this holds for every $n \in \N$, then consider the limit for $n$ to infinity. The strong convergences \eqref{EQ528} yield
	 \begin{align*}
		 \mathop{\lim\sup}_{\e\rightarrow 0} \frac{m_\e}{\e^5}\leq \lim_{n\rightarrow +\infty} \mathcal{J}\big(\Uc^{(n)},\wh{u}^{(n)}\big) = \mathcal{J}\big(\Uc^\prime,\wh{u}^\prime\big),
	 \end{align*}
	 which concludes the proof of \eqref{EQ516}.\\
\noindent{\it Step 3. } Hence, combining both steps we obtain for every $(\Uc^\prime,\wh{u}^\prime) \in \UU \times L^2(\o,H^1_{per}(\Yc^*))$
\begin{align}
	{\cal J}(\Uc,\wh{u}) \leq \liminf_{\e \rightarrow 0} \frac{m_\e}{\e^5}\leq \limsup_{\e \rightarrow 0} \frac{m_\e}{\e^5} \leq {\cal J}(\Uc^\prime, \wh{u}^\prime).
\end{align}
Thus, choosing $(\Uc^\prime,\wh{u}^\prime)=(\Uc,\wh{u})$ gives
$${\cal J}(\Uc,\wh{u}) = \lim_{\e \rightarrow 0} \frac{m_\e}{\e^5}$$
and finally, one obtains
\begin{align*}
	\lim_{\e\rightarrow 0} \frac{m_\e}{\e^5} = {\cal J}(\Uc,\wh{u}) = \min_{(\Uc^\prime, \wh{u}^\prime) \in \UU \times L^2(\o,H^1_{per}(\Yc^*))} {\cal J}(\Uc^\prime, \wh{u}^\prime).
\end{align*}
\vskip-11mm
\end{proof}

\subsection{The cell problems}\label{sec:cellprobl}
Recall the energy \eqref{limenergy}:
\begin{align}
{\cal J}(\Uc,\wh{u}) = \frac{1}{2}\int_\o \int_{\Yc^*} a\,\big(\GE(\Uc) + e_{y}(\widehat{u})\big):\big(\GE(\Uc) + e_{y}(\widehat{u})\big)\; dydx' -|\Yc^*| \int_\o f \cdot \Uc \; dx'.
\end{align}
\noindent To obtain the cell problems consider the variational formulation for $\wh{u}$ associated to the functional ${\cal J}$.  For this we use the Euler-Lagrange equation (since it is a quadratic form in $e(\wh{u})$ over a Hilbert-space) and we obtain:
	\begin{align}
	\begin{split}
	&\text{Find $\wh{u} \in L^2(\O;H^1(\Yc^*))^3$ such that}\\
	&\qquad \int_{\O\times \Yc^*} a \big(\GE(\Uc) + e_y(\widehat{u})\big): e_y(\widehat{w}) \, dy = 0,\qquad \text{for all } \widehat{w}\in L^2(\O;H^1_{per,0}(\Yc^*))^3. 
	\end{split}
	\end{align}
Upon this, we use the periodicity w.r.t. $\Yc^*$ to restrict the cell problems to
\begin{align}
\begin{split}
&\text{Find $\widehat{u} \in H^1_{per,0}(\Yc^*)^3$ such that }\\
&\qquad\int_{\Yc^*} a \big(\GE(\Uc) + e_y(\widehat{u})\big) :e_y(\widehat{w}) \; dy = 0,\qquad \text{for all } \widehat{w}\in H^1_{per,0}(\Yc^*)^3.
\end{split}
\end{align}
  Hence, the fields $\widehat{u}$ depend linearly on $\GE(\Uc)$, one has
\begin{align}
\widehat{u}(x',y) = \sum_{\alpha,\beta=1}^2 \Zc_{\alpha\beta}(x')\,\wh{\chi}^m_{\alpha\beta}(y) + \sum_{\alpha,\beta=1}^2 \partial_{\alpha\beta}\,\Uc_3(x')\widehat{\chi}^b_{\alpha\beta}(y).
\end{align}
This leads directly to the typical cell problems
\begin{align}\label{cell-problem}
\begin{split}
&\text{Find $\left(\widehat{\chi}^m_{11},\wh{\chi}^m_{12},\wh{\chi}^m_{22}, \wh{\chi}^b_{11}, \wh{\chi}^b_{12},\wh{\chi}^b_{22}\right) \in H^1_{per,0}(\Yc^*)^{3\times 6}$ such that }\\
&\qquad\left.
\begin{aligned}
&\int_{\Yc^*} a(y) (M^{\alpha\beta} + e_y(\wh{\chi}^m_{\alpha\beta})):e_y(\widehat{w}) \; dy = 0, \\
&\int_{\Yc^*} a(y) (-y_3M^{\alpha\beta} + e_y(\widehat{\chi}^b_{\alpha\beta})):e_y(\widehat{w}) \; dy = 0, 
\end{aligned}
\right\}
\text{ for all } \widehat{w}\in H^1_{per,0}(\Yc^*)^3,
\end{split}
\end{align}
where we denote
\begin{align}
	M^{11} = \begin{pmatrix}
	1&0&0\\0&0&0\\0&0&0
	\end{pmatrix},\qquad
	M^{22} = \begin{pmatrix}
	0&0&0\\0&1&0\\0&0&0
	\end{pmatrix},\qquad
	M^{12} = M^{21} = \begin{pmatrix}
	0&1&0\\1&0&0\\0&0&0
	\end{pmatrix}.
\end{align}
Then, set the homogenized coefficients
\begin{align}
\begin{split}\label{hom}
&a^{hom}_{\alpha\beta\alpha^\prime\beta^\prime} = \frac{1}{|\Yc^*|} \int_{\Yc^*} a_{ijkl}(y) \left[M^{\alpha\beta}_{ij} + 
e_{y,ij}(\wh{\chi}_{\alpha\beta}^m)\right]M^{\alpha^\prime\beta^\prime}_{kl} dy,\\
&b^{hom}_{\alpha\beta\alpha^\prime\beta^\prime} = \frac{1}{|\Yc^*|} \int_{\Yc^*} a_{ijkl}(y) \left[y_3M^{\alpha\beta}_{ij} + 
e_{y,ij}(\wh{\chi}_{\alpha\beta}^b)\right]M^{\alpha^\prime\beta^\prime}_{kl} dy,\\
&c^{hom}_{\alpha\beta\alpha^\prime\beta^\prime} = \frac{1}{|\Yc^*|} \int_{\Yc^*} a_{ijkl}(y) \left[y_3M^{\alpha\beta}_{ij} + 
e_{y,ij}(\wh{\chi}_{\alpha\beta}^b)\right]y_3M^{\alpha^\prime\beta^\prime}_{kl} dy.
\end{split}
\end{align}
Accordingly, the homogenized energy is defined by
\begin{equation}
\begin{aligned}\label{homo}
{\cal J}^{hom}_{vK} &(\Uc)={1\over |\Yc^*|}{\cal J}(\Uc,\wh{u}) \\
 &\frac{1}{2}\int_\o \big(a^{hom}_{\alpha\beta\alpha^\prime\beta^\prime}\Zc_{\alpha\beta}\Zc_{\alpha^\prime\beta^\prime} + 
b^{hom}_{\alpha\beta\alpha^\prime\beta^\prime}\Zc_{\alpha\beta}\;\partial_{\alpha^\prime\beta^\prime}\Uc_3 + 
c^{hom}_{\alpha\beta\alpha^\prime\beta^\prime}\partial_{\alpha\beta}\Uc_3\;\partial_{\alpha^\prime\beta^\prime}\Uc_3\big) \,dx' -\int_\o f\cdot \Uc \,dx'.
\end{aligned}
\end{equation}
with 
\begin{align*}
	\Zc_{\alpha\beta} = e_{\alpha\beta}(\Uc) + {1\over 2}\partial_\alpha \Uc_3\partial_\beta \Uc_3.
\end{align*}
\begin{thm}
	Under the assumptions \eqref{force}-\eqref{C2} the problem 
	\begin{align}
	\min_{\Uc \in \UU}  {\cal J}_{vK}^{hom} (\Uc)
	\end{align}
	admits solutions. Moreover, one has
$$m=\lim_{\e\rightarrow 0} \frac{m_\e}{\e^5} = \min_{(\Uc,\wh{u}) \in \UU \times L^2(\o ; H^1_{per}(\Yc^*))^3} \mathcal{J}(\Uc,\wh{u})=|\Yc^*|\min_{\Uc \in \UU}  {\cal J}_{vK}^{hom} (\Uc).$$
\end{thm}
\subsection{Yarn made of isotropic and homogeneous material}
 Let us here assume that the yarns are made from an isotropic and homogeneous material whose Lam\'e's constants are $\lambda$, $\mu$. The following Lemma shows that the homogenized textile is then 
 orthotropic.
		\begin{lemma} Under the above assumption on the material, one has
	\begin{equation}\label{EQ70}
	b^{hom}_{\alpha\beta\alpha^\prime\beta^\prime}=0\qquad \forall (\alpha,\beta,\alpha^\prime,\beta^\prime)\in\big\{1,2\big\}^4
	\end{equation}
	and also
	\begin{equation}\label{EQ700}
	\begin{aligned}
	&a^{hom}_{1111} =a^{hom}_{2222}\qquad \hbox{and}\qquad a^{hom}_{\alpha\alpha12}=0,\qquad \alpha\in \big\{1,2\big\},\\
	&c^{hom}_{1111} =c^{hom}_{2222}\qquad \hbox{and}\qquad c^{hom}_{\alpha\alpha12}=0,\qquad \alpha\in \big\{1,2\big\}
	\end{aligned}
	\end{equation}
\end{lemma}
\begin{proof}
	Consider the following transformation:
	$$
	\begin{aligned}
	&\phi\in H^1_{per}(\Yc^*)^3\;\longmapsto \; \wt\phi\in H^1_{per}(\Yc^*)^3\\
	& \wt\phi(y)=-\phi_1(\wt y)\Ge_1+\phi_2(\wt y)\Ge_2+\phi_3(\wt y)\Ge_3\quad \hbox{where}\quad \wt y= (2-y_1)\Ge_1+y_2\Ge_2+y_3\Ge_3,\qquad \hbox{for a.e. } y\in \Yc^*.
	\end{aligned}
	$$ One has
	$$
	\left\{\begin{aligned}
	&e_{y,ii}(\wt\phi)(y)=e_{y,ii}(\phi)(\wt y),\quad i\in\{1,2,3\},\\
	&e_{y,12}(\wt\phi)(y)=-e_{y,12}(\phi)(\wt y),\\
	&e_{y,13}(\wt\phi)(y)=-e_{y,13}(\phi)(\wt y),\\
	&e_{y,23}(\wt\phi)(y)=e_{y,23}(\phi)(\wt y),\\
	\end{aligned}\right.
	\qquad \hbox{for a.e. } y\in \Yc^*.
	$$
	Using this transformation in problem \eqref{cell-problem}$_2$  gives
	\begin{equation}\label{EQ71}
	\left\{\begin{aligned}
	&\wh{\chi}^b_{\alpha\alpha}(2-y_1,y_2,y_3)=\wh{\chi}^b_{\alpha\alpha}(y),\\
	&\wh{\chi}^b_{12}(2-y_1,y_2,y_3)=-\wh{\chi}^b_{12}(y),
	\end{aligned}\right.
	\qquad \hbox{for a.e. } y\in \Yc^*.
	\end{equation}
	Since $\ds\int_{\Yc^*} y_3\, dy=0$, one has
	$$b^{hom}_{\alpha\beta\alpha^\prime\beta^\prime} = \frac{1}{|\Yc^*|} \int_{\Yc^*} \sigma_{\alpha'\beta'} (\wh{\chi}_{\alpha\beta}^b)\,dy.$$
	Hence
	$$b^{hom}_{\alpha\alpha12}=0,\quad \alpha\in\{1,2\}.$$ 
	Now, from \eqref{EQ71}, we get
	\begin{equation}\label{EQ72}
	\left\{\begin{aligned}
	&\wh{\chi}^b_{12,1}(2-y_1,y_2,y_3)=\wh{\chi}^b_{12,1}(y),\\
	&\wh{\chi}^b_{12,2}(2-y_1,y_2,y_3)=-\wh{\chi}^b_{12,2}(y),\\
	&\wh{\chi}^b_{12,3}(2-y_1,y_2,y_3)=-\wh{\chi}^b_{12,3}(y),
	\end{aligned}\right.\qquad \left\{\begin{aligned}
	&\wh{\chi}^b_{\alpha\alpha,1}(2-y_1,y_2,y_3)=-\wh{\chi}^b_{\alpha\alpha,1}(y),\\
	&\wh{\chi}^b_{\alpha\alpha,2}(2-y_1,y_2,y_3)=\wh{\chi}^b_{\alpha\alpha,2}(y),\\
	&\wh{\chi}^b_{\alpha\alpha,3}(2-y_1,y_2,y_3)=\wh{\chi}^b_{\alpha\alpha,3}(y),
	\end{aligned}\right.
	\qquad \hbox{for a.e. } y\in \Yc^*.
	\end{equation} Equality \eqref{EQ72} and the periodicity lead to equalities below of the traces
	$$
	\begin{aligned}
	&\wh{\chi}^b_{12,i}(0,y_2,y_3)=\wh{\chi}^b_{12,i}(1,y_2,y_3)=\wh{\chi}^b_{12,i}(2,y_2,y_3)=0,\qquad i\in\{2,3\},\\
	&\wh{\chi}^b_{\alpha\alpha,1}(0,y_2,y_3)=\wh{\chi}^b_{\alpha\alpha,1}(1,y_2,y_3)=\wh{\chi}^b_{\alpha\alpha,1}(2,y_2,y_3)=0.
	\end{aligned}
	$$
	Now using the symmetry with respect to the plane $y_2=1$, we obtain
	\begin{equation}\label{EQ73}
	\left\{\begin{aligned}
	&\wh{\chi}^b_{\alpha\alpha}(y_1,2-y_2,y_3)=\wh{\chi}^b_{\alpha\alpha}(y),\\
	&\wh{\chi}^b_{12}(y_1,2-y_2,y_3)=-\wh{\chi}^b_{12}(y),
	\end{aligned}\right.
	\qquad \hbox{for a.e. } y\in \Yc^*.
	\end{equation}
	Hence
	$$
	\begin{aligned}
	&\wh{\chi}^b_{12,i}(y_1,0,y_3)=\wh{\chi}^b_{12,i}(y_1,1,y_3)=\wh{\chi}^b_{12,i}(y_1,2,y_3)=0,\qquad i\in\{1,3\},\\
	&\wh{\chi}^b_{\alpha\alpha,2}(y_1,0,y_3)=\wh{\chi}^b_{\alpha\alpha,2}(y_1,1,y_3)=\wh{\chi}^b_{\alpha\alpha,2}(y_1,2,y_3)=0.
	\end{aligned}
	$$ The  results above allow to replace problem \eqref{cell-problem}$_2$ by the following ones:
	\begin{equation}\label{cell-problem-2}
	\begin{aligned}
	&\left\{
	\begin{aligned}
	&\text{Find $\widehat{\chi}^b_{12} \in \GG(\Yc^*)$ such that }\\
	&\int_{\Yc^*} \sigma_{y,ii}(\widehat{\chi}^b_{12})\, e_{y,ij}(\widehat{w}) \, dy =\int_{\Yc^*} y_3\, \sigma_{y,12}(\widehat{w})\, dy, \\
	& \text{ for all } \widehat{w}\in \GG(\Yc^*)
	\end{aligned}
	\right.\\
	&\left\{
	\begin{aligned}
	&\text{Find $\widehat{\chi}^b_{\alpha\alpha} \in \GH(\Yc^*)$ such that }\\
	&\int_{\Yc^*} \sigma_{y,ii}(\widehat{\chi}^b_{\alpha\alpha})\, e_{y,ij}(\widehat{w}) \, dy =\int_{\Yc^*} y_3\, \sigma_{y,\alpha\alpha}(\widehat{w})\, dy, \\
	& \text{ for all } \widehat{w}\in \GH(\Yc^*)
	\end{aligned}
	\right.
	\end{aligned}
	\end{equation}
	where $\Yc^*$ is the part of the cell included in $(0,1)^2\times (-2\kappa,2\kappa)$ and ($i\in\{2,3\}$, $j\in\{1,3\}$)
	$$
	\begin{aligned}
	\GG(\Yc^*)=\big\{ \phi\in H^1(\Yc^*)^3\;|\; \phi_i(0,y_2,y_3)=\phi_i(1,y_2,y_3)=0,\quad \phi_j(y_1,0,y_3)=\phi_j(y_1,1,y_3)=0\big\},\\
	\GH(\Yc^*)=\big\{ \phi\in H^1(\Yc^*)^3\;|\; \phi_1(0,y_2,y_3)=\phi_1(1,y_2,y_3)=0,\quad \phi_2(y_1,0,y_3)=\phi_1(y_1,1,y_3)=0\big\}.
	\end{aligned}
	$$
	Now, consider the transformation
	$$
	\begin{aligned}
	&\phi\in \GH(\Yc^*)\;\longmapsto \;\overline{\phi}\in \GH(\Yc^*),\;\; \hbox{(resp. }\; \phi\in \GG(\Yc^*)\;\longmapsto \;\overline{\phi}\in \GG(\Yc^*)\hbox{)}\\
	&\overline{\phi}(y)=\phi_2(\overline{y})\Ge_1+\phi_1(\overline{y})\Ge_2-\phi_3(\overline{y})\Ge_3\quad \hbox{where}\quad  \overline{y}=y_2\Ge_1+y_1\Ge_2-y_3\Ge_3,\qquad \hbox{for a.e. } y\in \Yc^*.
	\end{aligned}
	$$ One has
	$$
	\begin{aligned}
	e_{y,11}(\overline{\phi})(y)&=e_{y,22}(\phi)(\overline{y}),&&e_{y,22}(\overline{\phi})(y)=e_{y,11}(\phi)(\overline{y}), && e_{y,33}(\overline{\phi})(y)=e_{y,33}(\phi)(\overline{y}),\\
	e_{y,13}(\overline{\phi})(y)&=-e_{y,23}(\phi)(\overline{y}), && e_{y,12}(\overline{\phi})(y)=e_{y,12}(\phi)(\overline{y}), && e_{y,23}(\overline{\phi})(y)=-e_{y,13}(\phi)(\overline{y}),
	\end{aligned}\quad \hbox{for a.e. } y\in \Yc^*.
	$$
	We use the above transformation in problems \eqref{cell-problem-2} that gives
	\begin{equation}\label{EQ74}
	\begin{aligned}
	&\wh{\chi}^b_{12}(y_2,y_1,-y_3)=-\wh{\chi}^b_{12}(y),\\
	&\wh{\chi}^b_{11}(y_2,y_1,-y_3)=-\wh{\chi}^b_{22}(y),
	\end{aligned}
	\qquad \hbox{for a.e. } y\in \Yc^*.
	\end{equation}
	These equalities lead to
	$$b^{hom}_{1212} =b^{hom}_{1122} =0,\qquad b^{hom}_{1111} =- b^{hom}_{2222}.$$
	The last transformation
	$$
	\begin{aligned}
	&\phi\in \GH(\Yc^*)\;\longmapsto \;\overline{\overline{\phi}}\in \GH(\Yc^*),\\
	&\overline{\overline{\phi}}(y)=-\phi_2(\overline{\overline{y}})\Ge_1+\phi_1(\overline{\overline{y}})\Ge_2+\phi_3(\overline{\overline{y}})\Ge_3\quad \hbox{where}\quad  \overline{\overline{y}}=(1-y_2)\Ge_1+y_1\Ge_2+y_3\Ge_3,\qquad \hbox{for a.e. } y\in \Yc^*.
	\end{aligned}
	$$ One has
	$$
	\begin{aligned}
	e_{y,11}(\overline{\overline{\phi}})(y)&=e_{y,22}(\phi)(\overline{\overline{y}}),&&e_{y,22}(\overline{\overline{\phi}})(y)=e_{y,11}(\phi)(\overline{\overline{y}}), && e_{y,33}(\overline{\overline{\phi}})(y)=e_{y,33}(\phi)(\overline{\overline{y}}),\\
	e_{y,13}(\overline{\overline{\phi}})(y)&=-e_{y,23}(\phi)(\overline{\overline{y}}), && e_{y,12}(\overline{\overline{\phi}})(y)=-e_{y,12}(\phi)(\overline{\overline{y}}), && e_{y,23}(\overline{\overline{\phi}})(y)=-e_{y,13}(\phi)(\overline{\overline{y}}),
	\end{aligned}\quad \hbox{for a.e. } y\in \Yc^*.
	$$
	We use the above transformation in problem \eqref{cell-problem-2}$_2$ that gives
	\begin{equation}\label{EQ78}
	\wh{\chi}^b_{11}(1-y_2,y_1,y_3)=\wh{\chi}^b_{22}(y),
	\qquad \hbox{for a.e. } y\in \Yc^*.
	\end{equation}
	This equality gives
	$$ b^{hom}_{1111} =b^{hom}_{2222}$$ which ends the proof of \eqref{EQ70}.
	Similarly one obtains \eqref{EQ700}.
\end{proof}
	As a consequence of the above lemma in the expressions of the energy \eqref{homo} and \eqref{homo2} they remain three coefficients $a^{hom}$ ($a^{hom}_{1111},\; a^{hom}_{1122},\; a^{hom}_{1212}$) and $c^{hom}$ ($c^{hom}_{1111},\; c^{hom}_{1122},\; c^{hom}_{1212}$).\\[1mm]

\subsection{The linear problem}

The analysis presented in this paper is stated especially for the von-K\'{a}rm\'{a}n limit. Although this is a nonlinear model the problem is stated with displacments and not deformations as usual in nonlinear elasticity. In fact the von-K\'{a}rm\'{a}n plate is the critical case for the choice of the geometric energy $\|dist(\nabla v,SO(3))\|_{L^2(\O_\e*)} \sim C\e^{5/2}$ in between linear and nonlinear plates, as it can be seen in \cite{Shell1,Shell2,Friesecke06,CiarletKarman,BCM,NV,FJMKarman}. \\[1mm]
To obtain the linear problem one simply considers the symmetric strain tensor $e(u)$ instead of the Green-Lagrangian strain tensor $e(u) + \frac{1}{2}\nabla u (\nabla u)^T = \frac{1}{2}(\nabla v (\nabla v)^T -\GI_3)$. All results in this paper remain true but the $\Zc_{\alpha\beta}(\Uc)$ are replaced by $e_{\alpha \beta}(\Uc)$ in the limit. 

The resulting linear limit energy 
\begin{align} \label{linear_limenergy}
{\cal J}_{lin}(\Uc,\widehat{u}) = \int_\o \int_{\Yc^*} \widehat{W}\big(y, \GE^{lin}(\Uc) + e_{y}(\widehat{u})\big)\; dydx' -|\Yc^*| \int_\o f \cdot \Uc \; dx',
\end{align}
with 
\begin{align}
	\GE^{lin}(\Uc) = 	\begin{pmatrix}
	e_{11}(\Uc) -y_3 \frac{\partial^2\Uc_3}{\partial x_1^2} & e_{12}(\Uc) -y_3 \frac{\partial^2\Uc_3}{\partial x_1\partial_2} & 0\\
	* & e_{22}(\Uc) -y_3 \frac{\partial^2\Uc_3}{\partial x_2^2} & 0\\
	0 & 0 & 0
	\end{pmatrix} 
\end{align}
found in \eqref{linstrainlimit}.

Then, with the same steps as in section \ref{sec:cellprobl} the cell problems are given by \eqref{cell-problem} and yield the homogenized linear plate equation also found in \cite[Thm. 11.21]{CDG}.

\begin{thm}\label{thhomlin}
	Assume that the force satisfies $f_\e = \e^{2+\nu}f_1\Ge_1 + \e^{2+\nu}f_2\Ge_2 + \e^{3+\nu}f_3\Ge_3$ with $f\in L^2(\o)$ and $\nu>0$. Then, ${\cal J}^{lin}$ is the unfolded limit energy. 
	Furthermore, the cell problems are again given by \eqref{cell-problem} and yield the homogenized energy
	\begin{multline}\label{homo2}
	{\cal J}^{hom}_{lin} (\Uc)=\\
	\frac{1}{2}\int_\o \big(a^{hom}_{\alpha\beta\alpha^\prime\beta^\prime}e_{\alpha\beta}(\Uc)e_{\alpha^\prime\beta^\prime}(\Uc) + 
	b^{hom}_{\alpha\beta\alpha^\prime\beta^\prime}e_{\alpha\beta}(\Uc)\partial_{\alpha^\prime\beta^\prime}\Uc_3 + 
	c^{hom}_{\alpha\beta\alpha^\prime\beta^\prime}\partial_{\alpha\beta}\Uc_3\partial_{\alpha^\prime\beta^\prime}\Uc_3 \big)dx' -\int_\o f\cdot \Uc\, dx'.
	\end{multline}
	The minimizer of this functional satisfies the variational problem
	\begin{multline}
		\text{Find $\Uc \in \UU$ such that for all $\Vc \in \UU$:}\\
			\int_\o a^{hom}_{\alpha\beta\alpha^\prime\beta^\prime}e_{\alpha\beta}(\Uc)e_{\alpha^\prime\beta^\prime}(\Vc) + \frac{b^{hom}_{\alpha\beta\alpha^\prime\beta^\prime}}{2}\left(e_{\alpha\beta}(\Uc)\;\partial_{\alpha^\prime\beta^\prime}\Vc_3 + e_{\alpha\beta}(\Vc)\partial_{\alpha^\prime\beta^\prime}\Uc_3\right)\\ 
	+c^{hom}_{\alpha\beta\alpha^\prime\beta^\prime}\partial_{\alpha\beta}\Uc_3\partial_{\alpha^\prime\beta^\prime}\Vc_3 dx' =\int_\o f\cdot \Vc\, dx'.
	\end{multline}
\end{thm}
Note that this is the same energy as for the problem presented in \cite[Ch. 11]{CDG} for the case $\theta = 1$.
The existence and uniqueness of a solution for this linear problem is for instance investigated in \cite{CDG,Ciarlet3,oleinik2009mathematical,ciarlet1997mathematical}.

\begin{remark}
	The shown derivation of a homogenized von-K\'{a}rm\'{a}n plate is also valid for other micro-structures for which the extension in section \ref{sec:prelim_ext} holds true, e.g. shells whose mid-surfaces are developable surfaces.
\end{remark}

\begin{remark}
	It is also possible to derive the von-K\'{a}rm\'{a}n plate on the level of deformations and the decomposition of deformation, see \cite{Shell1,Shell2}. However, this needs a more involved analysis of the decomposed fields, since there exist more degrees of freedom. This different approach yields some insights into nonlinear elasticity and the connection between nonlinear decomposition and linear decompositions (see also \cite{Shell2,GDecomp}), yet the result is the same as presented here.
\end{remark}

\section{Comparison to \cite{KT}}

In fact, the cell problems derived here are the same in \cite{KT}. However, it is not obvious on the first sight because of a different point of view therein. The coincidence of both cases can be explained by the fact, that the homogenization presented here is also valid for the problem stated in \cite{KT} with fixed junctions between the beams, i.e. the gap-function $g \equiv 0$, for linear elasticity. Besides replacing $\Zc_{\alpha\beta}$ by $e_{\alpha\beta}$ the result analogous to \eqref{homo}.

\section{Stability of a plate with von-K\'{a}rm\'{a}n energy}	

	Here, we give an example of the pre-strain in an orthotropic plate. 
	Linearization of the K{\'a}rm{\'a}n plate was considered in \cite{LM}, \cite{Puntel},\cite{berd}, under an assumption of scalar fields and that the in-plane elastic strain is equal to the given in-plane strain $ e^*$. 

	\subsubsection*{ Homogenization of the  in-plane pre-stress} \label{statprhom}
		In counterparts of \ref{stpr}, we can take  a given pre-stress in yarns, $ \ds{\sigma^*_{ij,\e}}(x):=\e^2 a_{ijkl}(\frac{x}{\varepsilon})  e^*_{kl,\e}(x)$, due to a thermal, chemical or electric expansion, with $e^*_{kl,\e}\in {L}^2(\O^*_\e)$, $(k,l)\in\{1,2,3\}^2$.
		We assume that 
		\begin{align}\label{limf}
		{\cal T}_{\e}\left(e^*_{\e}\right) \wc  e^* \quad \textrm{weakly in} \quad L^2(\omega\times \Yc^*;\R^{3\times 3}).
		\end{align}
	In \ref{stpr}, we replace the term $\ds \int_{\O^*_{\e}} f_{\e}\cdot(v-I_d)\, dx$ by the following:
	$$\int_{\O^*_{\e}} \e^2\; a\; e^*_{\e}\,:\, e(v-I_d)\, dx.$$	
	Then, as in Subsection \ref{SS53}, we prove that there exists a constant $C^*$ such that, if  $\|e^*_{\e}\|_{L^2(\O^*_\e)}\leq C^*\sqrt\e$ then for every $v\in \GV_{\e}$  such that $J_{\e}(v) \leq 0$ one has
	\begin{align*}
	\|dist(\nabla v, SO(3))\|_{L^2(\O^*_\e)} \leq C^{**} \e^{5/2}
	\end{align*}
	where the constants $C^*$ and $C^{**}$ do not depend on $\e$. Proceeding as in Section 6, we obtain the limit functional
	$$
	{\cal J}(\Uc,\wh{u}) = \frac{1}{2}\int_{\o\times \Yc^*} a\,\big(\GE(\Uc) + e_{y}(\widehat{u})\big):\big(\GE(\Uc) + e_{y}(\widehat{u})\big)\; dydx' -\frac{1}{2}\int_{\o\times 
	\Yc^*}a\; e^* : 
	\big(\GE(\Uc) + e_{y}(\widehat{u})\big)\, dx'dy.
	$$
	Now, we search the field $\widehat{u}$ with an additional corrector, responsible for the right-hand side, to homogenize the pre-stress, similar to  \cite{homel}. 
	Let $\widehat{\chi}^p$ be in  $L^2(\o;H^1_{per,0}(\Yc^*))^3$  the solution of
	$$
	\int_{\Yc^*}   a \,e_y(\wh{\chi}^p) : e_y(\widehat{w}) \; dy = \int_{\Yc^*} a\, e^*(x,\cdot) : e_y(\widehat{w}) \, dy,\qquad \text{for a.e. $x\in \o$ and for all } \widehat{w}\in 
	H^1_{per,0}(\Yc^*)^3.
	$$ One has
	\begin{align}
	\widehat{u}(x',y) = \sum_{\alpha,\beta=1}^2 \Zc_{\alpha\beta}(x') \wh{\chi}^m_{\alpha\beta}(y) + \sum_{\alpha,\beta=1}^2 
	\partial_{\alpha\beta}\,\Uc_3(x')\,\widehat{\chi}^b_{\alpha\beta}(y)+\wh{\chi}^p(x',y),\quad \hbox{for a.e. }(x',y)\in \o\times \Yc^*.
	\end{align}
Then, the effective pre-strain is
\begin{align}
&a_{\alpha \beta \alpha' \beta'}^{hom} (e_{\alpha \beta }^{*})^{hom}(x')= \frac{1}{|\cal \Yc^*|}\int_{\cal \Yc^*}   
a(y)   e^*(x',y) : \left({\mathbf M}^{\alpha' \beta'}+ e_{ y}(\wh{\chi}^p)(x',y)\right)  dy.
\end{align}

So, in te limit, the force functional will be   replaced by
\begin{equation}
|\Yc^*|\int_\omega a^{hom}(e^{*})^{hom}(x'):\Zc(x')\;dx',
\end{equation} 
with
\begin{align}
	\Zc = \begin{pmatrix}
	\Zc_{11} & \Zc_{12}\\
	* & \Zc_{22}
	\end{pmatrix},
\end{align}
Note that the macroscopic pre-stress will act in-plane, i.e. there will be no entry in the bending term from this pre-stress, because of the rotational symmetry of the periodicity cell.\\
Note, that since the minimization of a functional is up to a given additive term, the pre-strain can be insert into the left-hand side (energetic part) of the functional and substructed from the elastic strain.

\subsubsection*{Uni-axial compression}

 Now, all displacements must satisfy  the following inhomogeneous Dirichlet conditions:
		\begin{align}
		\Uc(0,x_2) = \begin{pmatrix}
		\ds{e^*L\over 2}\\[2mm] 0 \\[2mm] 0
		\end{pmatrix},\qquad \Uc(L,x_2) = \begin{pmatrix}
		\ds -{e^*L\over 2}\\[2mm] 0\\[2mm] 0
		\end{pmatrix}, \qquad \text{for a.e.  } x_2 \in (0,L).
		\end{align}
		Set
		\begin{align}
		\widetilde{\Uc}(x') = e^*\Big({L\over 2}-x_1\Big)\Ge_1\quad \hbox{for a.e. } x'=(x_1,x_2)\in \o.
		\end{align} This displacement satisfies the above Dirichlet conditions and one has 
		$$ e(\widetilde{\Uc})=\begin{pmatrix}
	-e^* & 0 \\
	    0 & 0
	\end{pmatrix}.$$
	Denote
		\begin{align*}
		\UU_{New}= \left\{ \Uc=\big(\Uc_1,\Uc_2,\Uc_3\big) \in  H^1(\o)^2 \times H^2(\o)\; \mid\; \Uc=0,\; \partial_1 \Uc_3 = 0\quad \hbox{a.e.  on } \Gamma_D \right\},\quad \Gamma_D=\{0,L\}\times (0,L).
		\end{align*}
		Now, our aim is to minimize over $\UU_{New}$ the functional
		\begin{equation}
		\begin{aligned}\label{homoNEW}
		{\cal J}^{hom}_{vK}  (\Uc+\widetilde{\Uc})&= \frac{1}{2}\int_\o \Big(a^{hom}_{1111}\big((\Zc'_{11})^2+(\Zc'_{22})^2\big)+ 4a^{hom}_{1212}(\Zc'_{12})^2 +2a^{hom}_{1122}\Zc'_{11}\Zc'_{22}\\
		&+ c^{hom}_{1111}\big((\partial_{11}\Uc_3)^2+(\partial_{22}\Uc_3)^2\big)+4c^{hom}_{1212} (\partial_{12}\Uc_3)^2+2c^{hom}_{1122} \partial_{11}\Uc_3\;\partial_{22}\Uc_3\Big)\, dx',\\
		&\quad \Uc\in \UU_{New} 
		\end{aligned}
		\end{equation}
		with here
		\begin{align*}
		\Zc'_{\alpha\beta} = e_{\alpha\beta}(\Uc+\widetilde{\Uc}) + \partial_\alpha \Uc_3\partial_\beta \Uc_3=\Zc_{\alpha\beta}+\Zc_{\alpha\beta}^*
		\end{align*} where
		$$\Zc^*_{\alpha\beta} = e_{\alpha\beta} (\widetilde{\Uc}) .$$
Again, we want to solve the minimization problem 
$$
\left\{\begin{aligned}
&\hbox{Find $\Uc^*\in \UU_{New}$ such that}\\
&\min_{\Uc\in \UU_{New}} {\cal J}^{hom}_{vK}  (\Uc+\widetilde{\Uc})={\cal J}^{hom}_{vK}  (\Uc^*+\widetilde{\Uc}).
\end{aligned}\right.
$$
We know that the infimum of this functional on $\UU_{New}$ is reached. Here, we want to get buckling. This means that the solutions of the above problem are not only in-plane displacements.

Suppose that there is no buckling. This means that the solutions of the above minimization problem are in fact the solution of the following minimization problem:
$$
\left\{\begin{aligned}
&\hbox{Find $\Uc^{*lin}\in \UU_{New}$ such that}\\
&\min_{\Uc\in \UU_{New}} {\cal J}^{hom}_{lin}  (\Uc+\widetilde{\Uc})={\cal J}^{hom}_{lin}  (\Uc^{*lin}+\widetilde{\Uc})
\end{aligned}\right.
$$ where
$$ 
{\cal J}^{hom}_{lin}  (\Uc+\widetilde{\Uc})= \frac{1}{2}\int_\o \Big(a^{hom}_{1111}\big((e_{11}(\Uc+\widetilde{\Uc}))^2+(e_{22}(\Uc+\widetilde{\Uc}))^2\big)+ 
4a^{hom}_{1212}(e_{12}(\Uc+\widetilde{\Uc}))^2 +2a^{hom}_{1122}e_{11}(\Uc+\widetilde{\Uc})e_{22}(\Uc+\widetilde{\Uc})\Big)\, dx'
$$
The above minimization problem admits a unique solution (an in-plane displacement 
$\Uc^{*lin}$). We have
$$0<{\cal J}^{hom}_{lin}  (\Uc^{*lin}+\widetilde{\Uc})=C^*(e^*)^2\leq {\cal J}^{hom}_{lin}  (\widetilde{\Uc})=  (e^*)^2a^{hom}_{1111}L^2\footnotemark.
$$
\footnotetext{ To get the exact value of the constant $C^*$ we have to solve the corresponding linear problem.}
Now, if we find a displacement of  type $\Vc_3\Ge_3$ such that
$${\cal J}^{hom}_{vK}  (\Vc_3\Ge_3+\widetilde{\Uc})<{\cal J}^{hom}_{lin}  (\Uc^{*lin}+\widetilde{\Uc})={\cal J}^{hom}_{vK}(\Uc^{*lin}+\widetilde{\Uc})$$  this will show that there is buckling.\\[1mm]
To do this, we seek a flexion independent of $x_2$ (taking into account the boundary conditions)
$$\Vc_3(x')=V_3(x_1)\quad \hbox{for a.e. } x'=(x_1,x_2)\in \o$$ with $V_3\in H^2_0(0,L)$ in order to get an admissible displacement ($\Vc_3\Ge_3\in \UU_{New}$). Hence, one has
$$
\begin{aligned}
&e_{11}(\widetilde{\Uc}) + (\partial_1 \Vc_3)^2=-e^*+(V'_3(x_1))^2,\\
&e_{12}(\widetilde{\Uc}) + \partial_1 \Vc_3\partial_2 \Uc_3= 0,\quad e_{22}(\widetilde{\Vc}) + \partial_2 \Vc_3\partial_2 \Uc_3=0,\\
&(\partial_{11}\Vc_3)^2=(V''_3(x_1))^2,\quad (\partial_{22}\Vc_3)^2=0,\\
& (\partial_{12}\Vc_3)^2=0,\quad  \partial_{11}\Vc_3\;\partial_{22}\Vc_3= 0.
\end{aligned}
$$ 
That gives
$$
\begin{aligned}
&{\cal J}^{hom}_{vK}  (\Vc_3\Ge_3+\widetilde{\Uc}) = \frac{L}{2}\int_0^L a^{hom}_{1111}\big((e^*)^2-2e^*(V'_3(x_1))^2+(V'_3(x_1))^4\big) + c^{hom}_{1111}(V''_3(x_1))^2dx_1.
\end{aligned}
$$
A necessary condition to obtain a buckling is ${\cal J}^{hom}_{vK}  (\Vc_3\Ge_3+\widetilde{\Uc}) < (e^*)^2a^{hom}_{1111}L^2$. Hence
$$\int_0^L a^{hom}_{1111}(V'_3(x_1))^4dx_1+\int_0^L c^{hom}_{1111}(V''_3(x_1))^2dx_1< 2e^*\int_0^L a^{hom}_{1111}(V'_3(x_1))^2dx_1.$$
Choose the function
$$V_3(x_1)=\sin^2\Big({\pi x_1\over L}\Big)\qquad\hbox{for all }x_1\in [0,L].$$
A straight forward calculation leads to
$$ a^{hom}_{1111} {3\pi^4\over 8L^3} + c^{hom}_{1111}{2\pi^4\over L^3} < 2e^* a^{hom}_{1111}{\pi^2\over 2L}.$$
Thus
$$e^* >{\pi^2 \over L^2}{3a^{hom}_{1111} +16 c^{hom}_{1111}\over  8 a^{hom}_{1111}}.$$
To get a sufficient condition we need to know a lower bound of $C^*$.
	\begin{remark}
	A weaker condition on $e^*$, which recovers conditions as in \cite{berd}, is given by the bound from below on the energy.
	\begin{align}
	&{\cal J}^{hom}_{vK}  (\Vc_3\Ge_3+\widetilde{\Uc})\geq\frac{L}{2}\int_0^L \left[ c^{hom}_{1111} -e^*\frac{L^2}{2\pi^2}a^{hom}_{1111} \right](V''_3(x_1))^2dx_1.
	\end{align} 
	where we used the Poincar\'e inequality 
	\begin{align}
	\int_0^L (V'_3(x_1))^2\;dx_1 \leq \frac{L^2}{(2\pi)^2} \int_{0}^{L} (V''_3(x_1))^2 \; dx_1.
	\end{align}
Then, we have
	\begin{align}
		(e^*)^2a^{hom}_{1111}L^2 > {\cal J}^{hom}_{vK}  (\Vc_3\Ge_3+\widetilde{\Uc}) \geq \frac{L}{2}\int_0^L \left[ c^{hom}_{1111} -e^*\frac{L^2}{2\pi^2}a^{hom}_{1111} \right](V''_3(x_1))^2dx_1
	\end{align}
A necessary condition to get buckling  is 
 $$\frac{L}{2}\int_0^L\left[ c^{hom}_{1111} -e^*\frac{L^2}{2\pi^2}a_{1111} \right](V''_3(x_1))^2dx_1 < 0.$$
That gives
	\begin{align}
	e^* > \frac{\pi^2c^{hom}_{1111}}{2 L^2 a^{hom}_{1111}}.
	\end{align}
	\end{remark}

For the sufficient condition, note that the coercivity yields
\begin{align}
	c \|e(\Uc^{*lin} - \wt\Uc)\|^2 \leq {\cal J}_{lin}^{hom} (\Uc^{*lin} - \wt{\Uc}).
\end{align}
Furthermore, note that $\|e(\Uc^{*lin} - \wt\Uc)\| > 0$ since the fields satisfy different boundary conditions. Indeed, we know that by the Korn-inequality, the trace-estimation
\begin{align}
	\|e(\Uc^{*lin} - \wt\Uc)\|_{L^2(\o)} \geq c \|\Uc^{*lin} - \wt\Uc\|_{H^1(\o)} \geq c\|\Uc^{*lin} - \wt\Uc\|_{L^2(\Gamma)} \geq c L^2e^* > 0
\end{align}
where $c$ does not  depend on $\o$.

\end{document}